\documentclass[a4paper]{amsart}
\pdfoutput=1 

\usepackage[utf8]{inputenc}
\usepackage[T1]{fontenc}
\usepackage{lmodern}
\usepackage{amssymb}
\usepackage{mathtools,enumitem}
\usepackage{microtype}
\usepackage{tikz,tikz-cd}

\tikzset{cd/.style=matrix of math nodes,row sep=2em,column sep=2em, text
height=1.5ex, text depth=0.5ex}
\tikzset{cdar/.style=->,auto}

\tikzset{dar/.style={double,double equal sign distance,-implies}}
\tikzset{mid/.style={anchor=mid}} 

\tikzset{triar/.style={anchor=mid,->}}
\tikzset{tridar/.style={anchor=mid,double,double equal sign
distance,-implies}}
\tikzset{narrowfill/.style={inner sep=0pt, fill=white}}

\setlist[enumerate,1]{label=\textup{(\arabic*)}}
\setlist[enumerate,2]{label=\textup{(\alph*)}}

\usepackage[pdftitle={The bicategory of groupoid correspondences},
pdfauthor={Celso Antunes, Joanna Ko, Ralf Meyer},
pdfsubject={Mathematics}
]{hyperref}

\usepackage[lite]{amsrefs}

\newcommand{\longref}[2]{\hyperref[#2]{#1~\textup{\ref*{#2}}}}

\numberwithin{equation}{section}

\theoremstyle{plain}
\newtheorem{theorem}[equation]{Theorem}
\newtheorem{lemma}[equation]{Lemma}
\newtheorem{proposition}[equation]{Proposition}
\newtheorem{deflem}[equation]{Definition and Lemma}
\newtheorem{corollary}[equation]{Corollary}

\theoremstyle{definition}
\newtheorem{definition}[equation]{Definition}

\theoremstyle{remark}
\newtheorem{remark}[equation]{Remark}
\newtheorem{example}[equation]{Example}

\newcommand*{\alb}{\hspace{0pt}} 
\newcommand*{\nb}{\nobreakdash}

\newcommand{\Comp}{\mathbb K}
\newcommand{\Bound}{\mathbb B}

\newcommand*{\Qu}{\mathsf{p}}
\newcommand*{\s}{s} 
\newcommand*{\rg}{r}

\newcommand*{\Cat}[1][C]{\mathcal #1}
\newcommand*{\Bisp}[1][X]{\mathcal #1}
\newcommand{\Gr}[1][G]{\mathcal #1}

\newcommand*{\prop}{\mathrm{prop}}
\newcommand*{\mult}{\mathrm{mult}}
\newcommand*{\Grcat}{\mathfrak{Gr}}
\newcommand{\Corr}{\mathfrak{Corr}}

\newcommand*{\ContS}{\mathfrak S}
\newcommand*{\Cont}{\mathrm C}
\newcommand*{\Contc}{\mathrm{C_c}}
\newcommand*{\Contb}{\mathrm{C_b}}

\newcommand*{\Hilm}[1][E]{\mathcal #1}
\newcommand*{\Hils}[1][H]{\mathcal #1}

\newcommand*{\Cst}{\texorpdfstring{\textup C^*}{C*}}
\newcommand*{\Star}{$^*$\nobreakdash-}

\newcommand*{\defeq}{\mathrel{\vcentcolon=}}
\newcommand*{\congto}{\xrightarrow\sim}

\newcommand*{\id}{\mathrm{id}}
\newcommand*{\pr}{\mathrm{pr}}

\newcommand{\C}{\mathbb{C}}
\newcommand{\N}{\mathbb{N}}

\newcommand{\Grcomp}{\circ}
\newcommand{\pt}{\mathrm{pt}}

\DeclarePairedDelimiter{\abs}{\lvert}{\rvert}
\DeclarePairedDelimiter{\norm}{\lVert}{\rVert}
\DeclarePairedDelimiterX{\braket}[2]{\langle}{\rangle}{#1\,\delimsize\vert\,\mathopen{}#2}%
\DeclarePairedDelimiterX{\braketop}[3]{\langle}{\rangle}{#1\,\delimsize\vert
#2\delimsize\vert\,\mathopen{}#3}
\DeclarePairedDelimiterX{\BRAKET}[2]{\langle}{\rangle}{\!\delimsize\langle#1\,\delimsize\vert\,\mathopen{}#2\delimsize\rangle\!}%
\DeclarePairedDelimiterX{\setgiven}[2]{\{}{\}}{#1\,{:}\,\mathopen{}#2}

\newcommand*{\conj}[1]{\overline{#1}}

\DeclareMathOperator{\supp}{supp}
\DeclareMathOperator{\Ad}{Ad}

\begin{document}
\title{The bicategory of groupoid correspondences}
\author{Celso Antunes}
\email{celso.antunes@mathematik.uni-goettingen.de}
\author{Joanna Ko}
\email{joanna.ko.maths@gmail.com}
\author{Ralf Meyer}
\email{rmeyer2@uni-goettingen.de}
\address{Mathematisches Institut\\
  Universität Göttingen\\
  Bunsenstraße 3--5\\
  37073 Göttingen\\
  Germany}

\begin{abstract}
  We define a bicategory with étale, locally compact groupoids as
  objects and suitable correspondences, that is, spaces with two
  commuting actions as arrows; the \(2\)\nb-arrows are injective,
  equivariant continuous maps.  We prove that the usual recipe for
  composition makes this a bicategory, carefully treating also
  non-Hausdorff groupoids and correspondences.  We extend
  the groupoid \(\Cst\)\nb-algebra construction to a homomorphism
  from this bicategory to that of \(\Cst\)\nb-algebra
  correspondences.  We describe the \(\Cst\)\nb-algebras of
  self-similar groups, higher-rank graphs, and discrete Conduché
  fibrations in our setup.
\end{abstract}


\maketitle

\section{Introduction}
\label{sec:intro}

Many interesting \(\Cst\)\nb-algebras may be realised as
\(\Cst\)\nb-algebras of étale, locally compact groupoids.  Examples
are the \(\Cst\)\nb-algebras associated to group actions on spaces,
(higher-rank) graphs, self-similar groups, and many
\(\Cst\)\nb-algebras associated to semigroups.  A (higher-rank)
graph is interpreted in~\cite{Albandik-Meyer:Product} as a
generalised dynamical system.  A self-similarity of a
group may also be interpreted in this way, namely, as a generalised
endomorphism of a group.  This suggests a way to put various
constructions of groupoids and their \(\Cst\)\nb-algebras under a
common roof, starting with a rather general kind of dynamical system.

This programme is worked out to a large extent in the dissertation
of Albandik~\cite{Albandik:Thesis}.  His results have not yet been
published in journal articles.  This article is concerned with the
most basic part of the programme.  Namely, we define groupoid
correspondences, which are the ``generalised maps'' between étale,
locally compact groupoids; we show that they form a bicategory, and
that taking groupoid \(\Cst\)\nb-algebras is a homomorphism to the
\(\Cst\)\nb-correspondence bicategory \(\Corr\) introduced
in~\cite{Buss-Meyer-Zhu:Higher_twisted}.

A homomorphism from a category to this bicategory~\(\Corr\) is
identified with a product system over~\(\Cat\)
in~\cite{Albandik-Meyer:Colimits}.  If the product system is proper,
then it gives rise to an ``absolute'' Cuntz--Pimsner algebra, where
the Cuntz--Pimsner covariance condition is asked for all elements.
We show that many constructions of \(\Cst\)\nb-algebras from
combinatorial or dynamical data are examples of such absolute
Cuntz--Pimsner algebras of product systems obtained from
homomorphisms to the bicategory of groupoid correspondences.  This
contains the \(\Cst\)\nb-algebras of regular topological graphs,
self-similar groups, row-finite higher-rank graphs, higher-rank
self-similar groups, and even the rather general discrete Conduché
fibrations of Brown and Yetter~\cite{Brown-Yetter:Conduche}.  Thus
the theory developed here offers a unified approach to several
important constructions of \(\Cst\)\nb-algebras.  In this article,
we only set up the bicategories and the homomorphism to~\(\Corr\)
and identify the resulting Cuntz--Pimsner algebras in some examples.
In following papers and in the thesis~\cite{Albandik:Thesis}, it is
shown how to realise these Cuntz--Pimsner algebras as groupoid
\(\Cst\)\nb-algebras, provided the underlying category satisfies Ore
conditions.  The relevance of the Ore conditions is also noticed in
the theory of discrete Conduché fibrations
in~\cite{Brown-Yetter:Conduche}.

While most results that we prove here are rather basic, there are
some technical difficulties that warrant a careful treatment.  It is
well known that mapping a groupoid to its groupoid
\(\Cst\)\nb-algebra cannot be functorial when we use functors as
arrows between groupoids.  The problem is manifest if we look at the
subclasses of spaces and groups: the group \(\Cst\)\nb-algebra is a
covariant functor for group homomorphisms, whereas the map
\(X\mapsto \Cont_0(X) = \Cst(X)\) for locally compact spaces is a
contravariant functor for proper continuous maps.  Buneci and
Stachura~\cite{Buneci-Stachura:Morphisms_groupoids} found a way
around this: they define suitable arrows between groupoids that do
induce morphisms between the groupoid \(\Cst\)\nb-algebras.  Our aim
are \(\Cst\)\nb-algebra correspondences instead of morphisms of
\(\Cst\)\nb-algebras.  Our theorem that there is a homomorphism of
bicategories from the bicategory of groupoid correspondences to that
of \(\Cst\)\nb-correspondences makes precise that the groupoid
\(\Cst\)\nb-algebra is ``functorial'' for these two types of
correspondences.  In the dissertation of Holkar (see
\cites{Holkar:Thesis, Holkar:Construction_Corr,
  Holkar:Composition_Corr}), a similar homomorphism is constructed
in the realm of Hausdorff, locally compact groupoids with Haar
system.  Holkar must decorate a groupoid correspondence with an analogue
of a Haar system, which makes his theory much more difficult.  To
reduce the technicalities, Holkar assumes his groupoids to be
Hausdorff.  We cannot do this, however, because the groupoids
associated to self-similar groups may fail to be Hausdorff, and we
want our theory to cover this case.

A groupoid correspondence is a space with commuting actions of the
two groupoids involved, which satisfy some extra conditions.  Asking
for more conditions, we get Morita equivalences of groupoids.  It is
well known that these form a bicategory and that taking groupoid
\(\Cst\)\nb-algebras is a homomorphism from this bicategory to the
bicategory of \(\Cst\)\nb-algebras and Morita--Rieffel equivalences;
this goes back already to the seminal work of
Muhly--Renault--Williams~\cite{Muhly-Renault-Williams:Equivalence},
except that they do not use the language of bicategories and allow
the more general case of locally compact groupoids with Haar
systems.  Another variant of groupoid correspondences was studied by
Hilsum--Skandalis~\cite{Hilsum-Skandalis:Morphismes} to construct
wrong-way functoriality maps between the K\nb-theory groups of
groupoid \(\Cst\)\nb-algebras.  These, however, usually fail to
induce \(\Cst\)\nb-correspondences.

Groupoid correspondences, Morita equivalences, and the morphisms of
Hilsum--Skandalis differ only in the technical details, that is, in
the extra conditions asked for the commuting actions of the two
groupoids.  The composition is defined in the same way in all three
cases.  What is different, of course, is the proof that the
composite again satisfies the relevant extra conditions.  Since the
bicategory of étale groupoid correspondences is a critical
ingredient in a larger programme, we find it useful to prove its
expected properties from scratch.

This article is structured as follows.  In
\longref{Section}{sec:groupoids}, we define étale groupoids and
their actions, and classes like free, proper, and basic actions.  We
also prove again that an action is free and proper if and only if it
is basic and its orbit space is Hausdorff.  In
\longref{Section}{sec:groupoid_corr}, we define étale groupoid
correspondences.  \longref{Section}{sec:groupoid_corr_examples}
illustrates them by several examples, which are related to
topological graphs, self-similar groups and self-similar graphs.
This justifies viewing groupoid correspondences as generalised maps
between groupoids.  In \longref{Section}{sec:composition}, we define
the composition of groupoid correspondences.  In
\longref{Section}{sec:bicategory}, we build a bicategory that has
groupoids as objects and groupoid correspondences as arrows.  We
also briefly recall the analogous bicategory of
\(\Cst\)\nb-correspondences.  In \longref{Section}{sec:to_Cstar}, we
build \(\Cst\)\nb-correspondences from groupoid correspondences and
show that this is part of a homomorphism of bicategories.  For
topological graphs and self-similar groups and graphs, we recover
\(\Cst\)\nb-correspondences that were used before to describe their
\(\Cst\)\nb-algebras as Cuntz--Pimsner algebras.  More generally, a
homomorphism from a monoid to the groupoid correspondence bicategory
gives rise to a product system over that monoid.  This suggests a
way to associate a \(\Cst\)\nb-algebra to any such homomorphism.  We
examine such homomorphisms and the resulting product systems at the
end of this article.  In \longref{Section}{sec:Conduche}, we
identify discrete Conduché fibrations with bicategory homomorphisms
into the bicategory of groupoid correspondences, and identify the
Cuntz--Pimsner algebra of the resulting product system with the
\(\Cst\)\nb-algebra of the discrete Conduché fibration as defined
previously.  We also briefly discuss the self-similar
\(k\)\nb-graphs of Li and Yang.

\section{Étale groupoids and groupoid actions}
\label{sec:groupoids}

Here we define (locally compact) étale groupoids and their actions
on topological spaces.  We prove that an action is free and proper
if and only if it is ``basic'' and has Hausdorff orbit space.  Most
of this is standard, and the last result is shown in
\cite{Buss-Meyer:Actions_groupoids}*{Proposition~A.7}.

\begin{definition}
  \label{def:groupoid}
  An \emph{étale \textup{(}topological\textup{)} groupoid} is a
  groupoid~\(\Gr\) with topologies on
  the arrow and object spaces \(\Gr\) and~\(\Gr^0\) such that the
  range and source maps \(\rg,\s\colon \Gr\rightrightarrows \Gr^0\)
  are local homeomorphisms and the multiplication and inverse maps are
  continuous.
  An étale groupoid is \emph{locally compact} if the object
  space~\(\Gr^0\) is Hausdorff and locally compact.
\end{definition}

We usually view~\(\Gr^0\) as a subset of~\(\Gr\) by the unit map,
that is, we identify an object \(x\in\Gr^0\) with the unit arrow
on~\(x\).

\begin{remark}
  We assume étale groupoids to be locally compact in order to pass
  to \(\Cst\)\nb-algebras later on.  The bicategory of groupoid
  correspondences may also be defined more generally, to have all
  étale groupoids as objects.  The reader interested in this will
  note that local compactness only becomes relevant in
  \longref{Section}{sec:to_Cstar} when we turn to \(\Cst\)\nb-algebras.
  Since all groupoids in this article shall be étale and locally
  compact, we usually drop these adjectives.
\end{remark}

\begin{definition}
  Let~\(\Gr\) be a groupoid.  A \emph{right \(\Gr\)\nb-space} is a
  topological space~\(\Bisp\) with a continuous map \(\s\colon
  \Bisp\to \Gr^0\), the \emph{anchor map}, and a continuous map
  \[
    \mult\colon \Bisp\times_{\s,\Gr^0,\rg} \Gr\to\Bisp,\qquad
    \Bisp\times_{\s,\Gr^0,\rg}\Gr\defeq
    \setgiven{(x,g)\in \Bisp\times \Gr}{\s(x)=\rg(g)},
  \]
  which we denote multiplicatively as~\(\cdot\), such that
  \begin{enumerate}
  \item \(\s(x\cdot g)=\s(g)\) for \(x\in\Bisp\), \(g\in\Gr\) with
    \(\s(x)=\rg(g)\);
  \item \((x\cdot g_1)\cdot g_2=x\cdot (g_1\cdot g_2)\) for \(x\in \Bisp\),
    \(g_1, g_2\in \Gr\) with \(\s(x)=\rg(g_1)\), \(\s(g_1)=\rg(g_2)\);
  \item \(x\cdot \s(x)=x\) for all \(x\in \Bisp\).
  \end{enumerate}
\end{definition}

\begin{definition}
  The \emph{orbit space}~\(\Bisp/\Gr\) is the
  quotient~\(\Bisp/{\sim_{\Gr}}\) with the quotient topology, where
  \(x\sim_{\Gr} y\) if there is an element \(g\in \Gr\) with
  \(\s(x)=\rg(g)\) and \(x\cdot g=y\).
\end{definition}

Left \(\Gr\)\nb-spaces are defined similarly.  We always write
\(\s\colon \Bisp\to \Gr^0\) for the anchor map in a right action and
\(\rg\colon \Bisp\to \Gr^0\) for the anchor map in a left action.

\begin{definition}
  \label{def:equivariant_map}
  Let \(\Bisp\) and~\(\Bisp[Y]\) be right \(\Gr\)\nb-spaces.  A
  continuous map \(f\colon \Bisp \to \Bisp[Y]\) is
  \emph{\(\Gr\)\nb-equivariant} if \(\s(f(x))=\s(x)\) for all
  \(x\in \Bisp\) and \(f(x\cdot g)=f(x)\cdot g\) for all \(x\in
  \Bisp\), \(g\in \Gr\) with \(\s(x)=\rg(g)\).
\end{definition}

\begin{definition}
  \label{def:invariant_map}
  Let \(\Bisp\) be a right \(\Gr\)\nb-space and~\(\Bisp[Z]\) a
  space.  A continuous map \(f\colon \Bisp \to \Bisp[Z]\) is
  \emph{\(\Gr\)\nb-invariant} if \(f(x\cdot g)=f(x)\) for all \(x\in
  \Bisp\), \(g\in \Gr\) with \(\s(x)=\rg(g)\).
\end{definition}

\begin{definition}
  \label{def:basic_action}
  A right \(\Gr\)\nb-space~\(\Bisp\) is \emph{basic} if the following
  map is a homeomorphism onto its image with the subspace topology
  from \(\Bisp\times\Bisp\):
  \begin{equation}
    \label{eq:action_map}
    f\colon \Bisp\times_{\s,\Gr^0,\rg} \Gr \to \Bisp\times\Bisp,\qquad
    (x,g)\mapsto (x\cdot g,x).
  \end{equation}
\end{definition}

The following useful lemma is used repeatedly throughout this article:

\begin{lemma}
  \label{lem:pullback_local_homeo}
  A pullback of a local homeomorphism is again a local
  homeomorphism.
\end{lemma}

\begin{proof}
  Consider the following pullback diagram with a continuous map
  $\alpha\colon A \to C$ and a local homeomorphism
  $\beta\colon B \to C$:
  \begin{equation}
    \label{pullback}
    \begin{tikzcd}
      A \times_{C} B \arrow[r, "\pi_2"] \arrow[d, "\pi_1"']
      \arrow[dr, phantom, "\scalebox{1.5}{$\lrcorner$}" , very near
      start, color=black]
      & B \arrow[d, "\beta"] \\
      A \arrow[r, "\alpha"'] & C
    \end{tikzcd}
  \end{equation}

  Let $(a, b) \in A \times_C B$.  By the definition of the product
  topology, any neighbourhood~\(N\) of~\((a,b)\) in $A \times_C B$
  contains a neighbourhood of the form \(U_a \times_C U_b\) with
  open neighbourhoods \(U_a\) and~\(U_b\) of \(a\) and~\(b\) in
  \(A\) and~\(B\), respectively.  Since~$\beta$ is a local
  homeomorphism, we may shrink~$U_b$ so that $\beta(U_b)$ is open
  and~$\beta$ restricts to a homeomorphism on~$U_b$.  Now
  \[
    \pi_1(U_a \times_C U_b)
    = \setgiven{x \in U_a}{\text{there is }y \in U_b \text{ with } \alpha(x)
      = \beta(y)}
    = \alpha^{-1}(\beta(U_b)) \cap U_a.
  \]
  Since~$\alpha$ is continuous, $\alpha^{-1}(\beta(U_b)) \cap U_a$ is open in~$A$.
  It follows that~\(\pi_1(N)\) is a neighbourhood of
  \(\pi_1(a,b)=a\).  Therefore, \(\pi_1\) is open.  Let
  $(a_1, b_1), (a_2, b_2) \in A \times_C U_b$ satisfy
  $\pi_1(a_1, b_1) = \pi_1(a_2, b_2)$.  Then $a_1 = a_2$ and
  $\beta(b_1) = \alpha(a_1) = \alpha(a_2) = \beta(b_2)$.  Since $\beta|_{U_b}$ is
  injective, this implies $b_1 = b_2$.  So~\(\pi_1|_{A\times_C
    U_b}\) is a homeomorphism onto an open subset of~\(A\).
\end{proof}

\begin{lemma}
  \label{lem:action_local_homeo}
  Let~$\Bisp$ be a right $\Gr$\nb-space.  The action
  $\mult\colon \Bisp \times_{\s, \Gr^0, \rg } \Gr \to \Bisp$ is a
  surjective local homeomorphism.
\end{lemma}

\begin{proof}
  The map
  \[
    (\mult, \id_{\Gr}) \colon \Bisp \times_{\s, \Gr^0, \rg} \Gr \to
    \Bisp \times_{\s,\Gr^0, \s} \Gr, \qquad (x, g) \mapsto (x \cdot g, g),
  \]
  is continuous because $\mult$ is continuous.  It has an inverse
  map
  \[
    (\mult', \id_{\Gr}) \colon \Bisp \times_{\s,\Gr^0, \s} \Gr \to \Bisp \times_{\s,
      \Gr^0, \rg} \Gr, \qquad (y, g) \mapsto (y \cdot g^{-1}, g),
  \]
  which is also continuous.  Hence $(\mult, \id_{\Gr})$ is a
  homeomorphism.  Since $\s \colon \Gr \to \Gr^0$ is a local
  homeomorphism, so is the coordinate projection
  $\pi_1 \colon \Bisp \times_{\s,\Gr^0, \s} \Gr \to \Bisp$ by
  \longref{Lemma}{lem:pullback_local_homeo}.  Then
  $\pi_1 \circ (\mult, \id_{\Gr}) = \mult$ is a local homeomorphism
  as the composite of two local homeomorphisms.  It is surjective
  because of the section
  \(\Bisp \to \Bisp \times_{\s, \Gr^0, \rg } \Gr\),
  \(x\mapsto (x,\s(x))\).
\end{proof}

\begin{lemma}
  \label{lem:basic_orbit_lh}
  The orbit space projection \(\Qu\colon \Bisp\to\Bisp/\Gr\) for a
  basic \(\Gr\)\nb-action is a surjective local homeomorphism.
\end{lemma}

\begin{proof}
  Let
  \begin{align*}
    \Bisp \times_{\s, \Gr^0,\rg} \Gr
    &\defeq
      \setgiven{(x, g) \in \Bisp \times \Gr}{\s(x) = \rg(g)},\\
    \Bisp \times_{p,\Bisp/\Gr,p} \Bisp
    &\defeq
      \setgiven{(x_1, x_2) \in \Bisp \times \Bisp}{p(x_1) = p(x_2)}.
  \end{align*}
  Since the right $\Gr$-action is basic, the following map is a
  homeomorphism:
  \[
    f\colon \Bisp \times_{\s, \Gr^0, \rg } \Gr \to
    \Bisp \times_{p,\Bisp/\Gr,p} \Bisp, \qquad
    (x, g) \mapsto (x \cdot g, x).
  \]
  The set of identity arrows $\Gr^0 \subseteq \Gr^1$ is open
  because~\(\Gr\) is étale.  Then $\Bisp \times \Gr^0$ is open in
  $\Bisp \times \Gr$.  Hence
  $I \defeq (\Bisp \times \Gr^0) \cap (\Bisp \times_{\s, \Gr^0, \rg}
  \Gr)$ is open in $\Bisp \times_{\s, \Gr^0, \rg} \Gr$.  If
  \((x,g)\in I\), then \(f(x,g) = x \cdot \s(x) = x\).  Therefore,
  $f(I) = \setgiven{(x, x)}{x \in X}$.  Since $f$ is a
  homeomorphism, $f(I)$ is open in
  $\Bisp \times_{p,\Bisp/\Gr,p} \Bisp$.  Thus there is an open
  subset $V \subseteq \Bisp \times \Bisp$ such that
  $f(I) = V \cap \Bisp \times_{p,\Bisp/\Gr,p} \Bisp$.  For each
  \(x\in \Bisp\), there is an open neighbourhood $U \subseteq \Bisp$
  with $U \times U \subseteq V$.  Then
  \[
    (U \times U) \cap (\Bisp \times_{p,\Bisp/\Gr,p} \Bisp)
    = \setgiven{(u, u)}{u \in U}.
  \]
  Suppose $p(u_1) = p(u_2)$ for some $u_1, u_2 \in U$.  Then
  $u_1 = u_2 \cdot g$ for some $g \in \Gr$.  Then
  $(u_2 \cdot g, u_2) \in (U \times U) \cap (\Bisp
  \times_{p,\Bisp/\Gr,p} \Bisp)$ and hence $u_1 = u_2$.  That is,
  $p$ is injective on~$U$.  Next we show that~$p$ is open.  Let
  $W \subseteq \Bisp$ be open.  Then
  \[
    p^{-1}(p(W))
    = \setgiven{w \cdot g}{w \in W,\ g \in \Gr,\ \s(w) = \rg(g)}
    = \mult(W \times_{\s, \Gr^0, \rg} \Gr)
  \]
  is open by \longref{Lemma}{lem:action_local_homeo}.  Since~$p$ is a quotient
  map, it follows that~$p(W)$ is open.
\end{proof}

\begin{definition}
  \label{def:free_action}
  A right \(\Gr\)\nb-space is \emph{free} if the map
  in~\eqref{eq:action_map} is injective; equivalently,
  \(x\cdot g = x\) for \(x\in \Bisp\), \(g\in \Gr\) with
  \(\s(x)=\rg(g)\) implies \(g = \s(x)\).
\end{definition}

\begin{definition}
  \label{def:proper_action}
  A continuous map~\(f\) is \emph{proper} if the map
  \(f\times \id_Z\) is closed for any topological space~\(Z\).  A
  right \(\Gr\)\nb-space~\(\Bisp\) is \emph{proper} if the map
  in~\eqref{eq:action_map} is proper.  An étale groupoid is proper
  if its canonical action on~\(\Gr^0\) is proper.  Equivalently, the
  following map is proper:
  \[
    (\rg,\s)\colon\Gr\to\Gr^0\times\Gr^0,\qquad
    g\mapsto (\rg(g),\s(g)).
  \]
\end{definition}

Our next goal is to relate free and proper actions to basic actions.

\begin{lemma}
  \label{lem:closed_vs_Hausdorff}
  For a basic $\Gr$\nb-action,
  $\Bisp \times_{p,\Bisp/\Gr,p} \Bisp \subseteq \Bisp \times \Bisp$
  is closed if and only if~$\Bisp/\Gr$ is Hausdorff.
\end{lemma}

\begin{proof}
  Assume first that
  $\Bisp \times_{p,\Bisp/\Gr,p} \Bisp \subseteq \Bisp \times \Bisp$
  is closed.  Choose $y_1 \ne y_2$ in $\Bisp/\Gr$.  There are
  $x_1, x_2 \in \Bisp$ with $p(x_1) = y_1$ and $p(x_2) = y_2$.  Then
  $(x_1, x_2) \notin \Bisp \times_{p,\Bisp/\Gr,p} \Bisp$ because
  $x_1$ and~$x_2$ are not in the same orbit.  Since
  $\Bisp \times_{p,\Bisp/\Gr,p} \Bisp$ is closed, its complement is
  open.  This gives open neighbourhoods $U_1 \ni x_1$ and
  $U_2 \ni x_2$ with
  $(U_1 \times U_2) \cap (\Bisp \times_{p,\Bisp/\Gr,p} \Bisp) =
  \emptyset$.  Then $p(U_1) \cap p(U_2) = \emptyset$ by definition
  of $\Bisp \times_{p,\Bisp/\Gr,p} \Bisp$.  Since the
  $\Gr$\nb-action is basic, $p$ is open by
  \longref{Lemma}{lem:basic_orbit_lh}.  Hence $p(U_1)$ and $p(U_2)$ are open
  neighbourhoods that separate $y_1$ and $y_2$.  This shows
  that~$\Bisp/\Gr$ is Hausdorff.

  Conversely, assume that~$\Bisp/\Gr$ is Hausdorff.  Let
  $(x_1, x_2) \in (\Bisp \times \Bisp) \setminus (\Bisp
  \times_{p,\Bisp/\Gr,p} \Bisp)$.  Then $p(x_1) \neq p(x_2)$.  Since
  $\Bisp/\Gr$ is Hausdorff, there are open neighbourhoods
  $V_1 \ni p(x_1)$ and $V_2 \ni p(x_2)$ with
  $V_1 \cap V_2 = \emptyset$.  Then $p^{-1}(V_1) \times p^{-1}(V_2)$
  is an open neighbourhood of $(x_1, x_2)$ that does not meet
  $\Bisp \times_{p,\Bisp/\Gr,p} \Bisp$.  This shows that
  $\Bisp \times_{p,\Bisp/\Gr,p} \Bisp$ is closed in $\Bisp \times \Bisp$.
\end{proof}

\begin{lemma}
  \label{lem:free_proper_closed}
  For a free and proper action,
  $\Bisp \times_{p,\Bisp/\Gr,p} \Bisp \subseteq \Bisp \times \Bisp$
  is closed.
\end{lemma}

\begin{proof}
  The composite map
  \begin{equation}
    \label{eq:definition_of_F}
    \begin{tikzcd}
      F\colon \Bisp \times_{\s, \Gr^0, \rg }
      \Gr \rar["f"]& \Bisp \times_{p,\Bisp/\Gr,p}
      \Bisp \rar[hook]& \Bisp \times \Bisp
    \end{tikzcd}
  \end{equation}
  is proper.  Hence \(F(\Bisp \times_{\s, \Gr^0, \rg} \Gr)\) is
  closed.  This is equal to
  \(\Bisp \times_{p,\Bisp/\Gr,p} \Bisp\) because~\(f\) is
  bijective.
\end{proof}

\begin{lemma}
  \label{lem:basic_to_proper}
  For a basic action, if
  $\Bisp \times_{p,\Bisp/\Gr,p} \Bisp \subseteq \Bisp \times \Bisp$
  is closed, then the map~$F$ defined in~\eqref{eq:definition_of_F}
  is proper.
\end{lemma}

\begin{proof}
  By definition, \(F\) is proper if and only if the map
  \(F\times \id_Z\) is closed for any topological space~\(Z\).
  Since the action is basic, $F$ is a homeomorphism onto its image,
  which is closed by assumption.  The property of being a
  homeomorphism onto a closed subset is preserved by taking products
  with any space~\(Z\).  Thus $F \times \id_Z$ is a closed map.
\end{proof}

\begin{proposition}
  \label{pro:basic_actions_free_proper}
  Let~\(\Gr\) be an étale groupoid and~\(\Bisp\) a right \(\Gr\)\nb-space.
  The following are equivalent:
  \begin{enumerate}
  \item the action of~\(\Gr\) on~\(\Bisp\) is basic and the orbit
    space~\(\Bisp/\Gr\) is Hausdorff;
  \item the action of~\(\Gr\) on~\(\Bisp\) is free and proper.
  \end{enumerate}
\end{proposition}

\begin{proof}
  Suppose first that the $\Gr$\nb-action is basic and $\Bisp/\Gr$ is
  Hausdorff.  Since
  $f\colon \Bisp \times_{\s, \Gr^0, \rg} \Gr \to \Bisp
  \times_{p,\Bisp/\Gr,p} \Bisp$ is a homeomorphism, it must be
  injective.  That is, the $\Gr$\nb-action is free.  By
  \longref{Lemma}{lem:closed_vs_Hausdorff},
  $\Bisp \times_{p,\Bisp/\Gr,p} \Bisp \subseteq \Bisp \times \Bisp$
  is closed.  Then~$F$ is proper by \longref{Lemma}{lem:basic_to_proper}.
  That is, the $\Gr$\nb-action is proper.

  Conversely, suppose the $\Gr$\nb-action to be free and proper.  We
  first show that the $\Gr$\nb-action is basic.  Since~$f$ is
  clearly bijective, it remains to show that~$f^{-1}$ is continuous.
  Let~$U$ be open in $\Bisp \times_{\s, \Gr^0, \rg } \Gr$.  Then
  $(f^{-1})^{-1}(U) = f(U)$ is open since~$f$ is closed and
  bijective, so that~$f$ is open.  By
  \longref{Lemma}{lem:free_proper_closed},
  $\Bisp \times_{p,\Bisp/\Gr,p} \Bisp \subseteq \Bisp \times \Bisp$
  is closed.  Hence $\Bisp/\Gr$ is Hausdorff by
  \longref{Lemma}{lem:closed_vs_Hausdorff}.
\end{proof}

\section{Groupoid correspondences}
\label{sec:groupoid_corr}

\begin{definition}
  \label{def:groupoid_corr}
  Let \(\Gr[H]\) and~\(\Gr\) be étale groupoids.  An
  \textup{(}étale\textup{)} \emph{groupoid correspondence}
  from~\(\Gr\) to~\(\Gr[H]\), denoted
  \(\Bisp\colon \Gr[H]\leftarrow \Gr\), is a space~\(\Bisp\) with
  commuting actions of~\(\Gr[H]\) on the left and~\(\Gr\) on the
  right, such that the right anchor map \(\s\colon \Bisp\to \Gr^0\)
  is a local homeomorphism and the right \(\Gr\)\nb-action is
  free and proper.

  That the actions of \(\Gr[H]\) and~\(\Gr\) commute means that
  \(\s(h\cdot x)=\s(x)\), \(\rg(x\cdot g) = \rg(x)\), and \((h\cdot x)
  \cdot g = h\cdot (x\cdot g)\) for all \(g\in \Gr\), \(x\in\Bisp\),
  \(h\in \Gr[H]\) with \(\s(h)=\rg(x)\) and \(\s(x)=\rg(g)\), where
  \(\s\colon \Bisp\to\Gr^0\) and \(\rg\colon \Bisp\to\Gr[H]^0\) are
  the anchor maps.
\end{definition}

\begin{remark}
  \label{rem:groupoid_corr_locally_compact}
  Since~\(\Gr^0\) is locally compact and~\(\s\) is a local
  homeomorphism, the underlying space~\(\Bisp\) of a groupoid
  correspondence is locally compact as well.  The space~\(\Bisp\)
  itself need not be Hausdorff.
  \longref{Proposition}{pro:basic_actions_free_proper} implies, instead,
  that the space~\(\Bisp/\Gr\) is Hausdorff.
\end{remark}

\begin{definition}
  A correspondence \(\Bisp\colon \Gr[H]\leftarrow \Gr\) is
  \emph{proper} if the map \(\rg_*\colon \Bisp/\Gr\to \Gr[H]^0\)
  induced by~\(\rg\) is proper.  It is \emph{tight} if~$\rg_*$ is a
  homeomorphism.
\end{definition}

\begin{deflem}
  \label{def:innerprodofcorr}
  Let \(\Bisp\colon \Gr[H]\leftarrow \Gr\) be a groupoid
  correspondence.  Let \(\Qu\colon \Bisp\to \Bisp/\Gr\) be the orbit
  space projection.  The image of the map~\eqref{eq:action_map} is
  the subset
  \(\Bisp \times_{\Bisp/\Gr} \Bisp = \Bisp
  \times_{\Qu,\Bisp/\Gr,\Qu} \Bisp\) of all
  \((x_1,x_2)\in\Bisp\times\Bisp\) with \(\Qu(x_1)=\Qu(x_2)\).  The
  inverse to the map in~\eqref{eq:action_map} induces a continuous
  map
  \begin{equation}
    \label{eq:innprod_correspondence}
    \Bisp \times_{\Bisp/\Gr} \Bisp \congto
    \Bisp \times_{\s,\Gr^0,\rg} \Gr \xrightarrow{\pr_2} \Gr,\qquad
    (x_1,x_2)\mapsto \braket{x_2}{x_1}.
  \end{equation}
  That is, \(\braket{x_1}{x_2}\) is defined for \(x_1,x_2\in\Bisp\)
  with \(\Qu(x_1)=\Qu(x_2)\) in~\(\Bisp/\Gr\), and it is the unique
  \(g\in\Gr\) with \(\s(x_1) = \rg(g)\) and \(x_2 = x_1 g\).  The
  right \(\Gr\)\nb-action on~\(\Bisp\) is basic if and only if
  \(g\in\Gr\) with \(x_2= x_1 g\) for \(x_1,x_2\in\Bisp\) with
  \(\Qu(x_1)=\Qu(x_2)\) is unique and the resulting map
  \(\Bisp\times_{\Bisp/\Gr}\Bisp\to \Gr\),
  \((x_1,x_2)\mapsto\braket{x_1}{x_2}\), is continuous.
\end{deflem}

\begin{proof}
  The inverse to the map in~\eqref{eq:action_map} is of the form
  \[
    \Bisp\times_{\Bisp/\Gr}\Bisp \to \Bisp\times_{\s,\Gr^0,\rg}
    \Gr,\qquad
    (x_1,x_2)\mapsto (x_2,\braket{x_2}{x_1}).
  \]
  This is continuous if and only if the map
  in~\eqref{eq:innprod_correspondence} is continuous.
\end{proof}

The following proposition says that $\braket{x_1}{x_2}$ has
properties analogous to those of rank-one operators on Hilbert
modules, which justifies our notation.

\begin{proposition}
  \label{pro:innprod}
  Let \(\Bisp\colon \Gr[H]\leftarrow\Gr\) be a groupoid
  correspondence.  The map in~\eqref{eq:innprod_correspondence} is a
  local homeomorphism.  It has the following properties:
  \begin{enumerate}
  \item \label{en:innprod_1}%
    \(\rg(\braket{x_1}{x_2}) = \s(x_1)\),
    \(\s(\braket{x_1}{x_2}) = \s(x_2)\), and
    \(x_2 = x_1\cdot \braket{x_1}{x_2}\) for all \(x_1,x_2\in\Bisp\)
    with \(\Qu(x_1)= \Qu(x_2)\);
  \item \label{en:innprod_2}%
    \(\braket{x}{x} = 1\) for all \(x\in\Bisp\);
  \item \label{en:innprod_3}%
    \(\braket{x_1}{x_2} = \braket{x_2}{x_1}^{-1}\) for all
    \(x_1,x_2\in\Bisp\) with \(\Qu(x_1)= \Qu(x_2)\);
  \item \label{en:innprod_4}%
    \(\braket{h x_1 g_1}{h x_2 g_2} = g_1^{-1}
    \braket{x_1}{x_2} g_2\) for all \(h\in\Gr[H]\),
    \(x_1,x_2\in\Bisp\), \(g_1,g_2\in\Gr\) with
    \(\s(h) = \rg(x_1) = \rg(x_2)\),
    \(\s(x_1)=\rg(g_1)\), \(\s(x_2)=\rg(g_2)\),
    \(\Qu(x_1)=\Qu(x_2)\).
  \end{enumerate}
\end{proposition}

\begin{proof}
  Since~$\s$ is a local homeomorphism, so is the second coordinate
  projection
  $\pi_2 \colon \Bisp \times_{\s, \Gr^0, \rg} \Gr \to \Gr$ by
  \longref{Lemma}{lem:pullback_local_homeo}.  This map composed with
  the homeomorphism
  \(\Bisp\times_{\Bisp/\Gr}\Bisp\cong\Bisp \times_{\s, \Gr^0, \rg}
  \Gr\) is the bracket map.  The properties of~\(\braket{x_1}{x_2}\)
  are checked by direct computations, using that for any
  \((x_1,x_2) \in \Bisp\times_{\Bisp/\Gr}\Bisp\), there is only one
  \(g\in\Gr\) with \(x_1\cdot g = x_2\), namely,
  $g=\braket{x_1}{x_2}$.  This equation forces $\rg(g) = \s(x_1)$
  and $\s(g) = \s(x_1 \cdot g) = \s(x_2)$, which
  gives~\ref{en:innprod_1}.  Since $x = x \cdot 1_{\s(x)}$, it
  implies~\ref{en:innprod_2}.  Since $x_1 = x_2 \cdot g$ if and only
  if $x_1 \cdot g^{-1} = x_2$, it implies~\ref{en:innprod_3}.  Since
  \((h x_1 g_1)\cdot (g_1^{-1} \braket{x_1}{x_2} g_2) = h x_1
  \braket{x_1}{x_2} g_2 = h x_2 g_2\), it
  implies~\ref{en:innprod_4}.
\end{proof}

\section{Some examples of groupoid correspondences}
\label{sec:groupoid_corr_examples}

In this section, we examine our definition of a groupoid
correspondence when both groupoids \(\Gr\) and~\(\Gr[H]\) are
locally compact spaces, discrete groups, or transformation groups.
We get topological graphs, self-similarities of groups, and
self-similarities of graphs in these three cases, respectively;
these objects have been used before in order to define
\(\Cst\)\nb-algebras.  More precisely, the self-similarities of
groups and graphs correspond to \emph{proper} groupoid
correspondences on groups and transformation groups, respectively.

Any locally compact space gives a groupoid with only identity
arrows.  We first describe groupoid correspondences between such
groupoids.

\begin{example}
  \label{exa:topological_graph}
  Let \(\Gr\) and~\(\Gr[H]\) be locally compact spaces, viewed as
  groupoids with only identity arrows.  A groupoid action of \(\Gr\)
  or~\(\Gr[H]\) is simply a continuous map to these two spaces.  The
  orbit space of an action on a space~\(\Bisp\) is again~\(\Bisp\).
  Any such action is basic.  By
  \longref{Proposition}{pro:basic_actions_free_proper}, the underlying
  space of a groupoid correspondence must be locally compact and
  Hausdorff.  Summing up, a groupoid correspondence
  \(\Bisp\colon \Gr[H]\leftarrow \Gr\) is the same as a locally
  compact, Hausdorff space~\(\Bisp\) with a continuous map
  \(\rg\colon \Bisp\to \Gr[H]\) and a local homeomorphism
  \(\s\colon \Bisp\to \Gr\).  The correspondence is proper if and
  only if \(\rg\colon \Bisp\to \Gr[H]\) is proper, and tight if and
  only if \(\rg\colon \Bisp\to\Gr[H]\) is a homeomorphism.  In the
  tight case, we may use~\(\rg\) to identify \(\Bisp=\Gr[H]\).  This
  gives an isomorphic groupoid correspondence with
  \(\rg=\id_{\Gr[H]}\).  Thus a tight groupoid correspondence
  \(\Gr[H]\leftarrow \Gr\) is equivalent to a local homeomorphism
  \(\Gr[H]\to \Gr\).

  A groupoid correspondence \((\Bisp,\rg,\s)\) as above with locally
  compact, Hausdorff \(\Bisp\), \(\Gr\) and~\(\Gr[H]\) is called a
  topological correspondence in~\cite{Albandik-Meyer:Product}.  If,
  in addition, \(\Gr[H]=\Gr\), then it is called a topological graph
  in~\cite{Katsura:class_I}; this is the data from which topological
  graph \(\Cst\)\nb-algebras are built.  Similar notions are
  also introduced by Deaconu~\cite{Deaconu:Continuous_graphs} and
  Nekrashevych~\cite{Nekrashevych:Combinatorial_models}, under the
  names ``continuous graph'' or ``topological automaton.''  These
  notions are meant to generalise non-invertible dynamical systems.

  There are two different ways to turn a local homeomorphism
  \(f\colon \Gr\to\Gr\) into a groupoid correspondence, namely,
  \(\rg=f\) and \(\s=\id_{\Gr}\) or \(\rg=\id_{\Gr}\) and \(\s=f\).
  Unless~\(f\) is a homeomorphism, the resulting topological graph
  \(\Cst\)\nb-algebras are not closely related.  So these two
  constructions must be distinguished carefully.
\end{example}

Next we describe groupoid correspondences between groups.

\begin{example}
  \label{exa:self-similar_groups}
  Let \(\Gr[H]\) and~\(\Gr\) be discrete groups.  A groupoid
  correspondence \(\Gr[H]\leftarrow \Gr\) is a space~\(\Bisp\) with
  commuting actions of~\(\Gr[H]\) on the left and~\(\Gr\) on the
  right, such that the right action is basic with Hausdorff orbit
  space and the right anchor map is a local homeomorphism.
  Since~\(\Gr^0\) is the one-point set, the anchor map
  \(\s\colon \Bisp\to\Gr^0\) is a local homeomorphism if and only
  if~\(\Bisp\) is discrete.  In this case, the right
  \(\Gr\)\nb-action is basic if and only if it is free.  Thus a
  groupoid correspondence \(\Bisp\colon \Gr[H]\leftarrow \Gr\) is a
  set~\(\Bisp\) with the discrete topology and with commuting
  actions of~\(\Gr[H]\) on the left and~\(\Gr\) on the right, where
  the right action is free.

  Let \(A\defeq \Bisp/\Gr\).  Since the right action is free, there is
  a bijection \(A\times \Gr\cong \Bisp\) such that the right
  \(\Gr\)\nb-action on~\(\Bisp\) becomes \((x,g_1)\cdot g_2
  = (x,g_1 g_2)\) on~\(A\times \Gr\).  We transfer the left
  \(\Gr[H]\)\nb-action to~\(A\times \Gr\) using this bijection.  Thus
  \(A\times \Gr\) becomes a groupoid correspondence \(\Gr[H]\leftarrow
  \Gr\) that is isomorphic to~\(\Bisp\).  The left action of \(h\in
  \Gr[H]\) on~\(A\times \Gr\) must be of the form
  \[
    h\cdot (x,g) = (\pi_h(x),h|_x\cdot g)
  \]
  for some map \(\pi_h\colon A\to A\) and some \(h|_x\in \Gr\)
  because it commutes with the right \(\Gr\)\nb-action.  The property
  \((h_1\cdot h_2)\cdot (x,g)= h_1\cdot (h_2\cdot (x,g))\) of a left
  action is equivalent to the following two conditions.  First,
  \(h\mapsto \pi_h\) must be a group action of~\(\Gr[H]\) on~\(A\).
  Secondly, the map \(\varphi\colon \Gr[H]\times A\to\Gr\),
  \((h,x)\mapsto h|_x\), must be a \(1\)\nb-cocycle, that is,
  \[
    (h_1 h_2)|_x = h_1|_{h_2 x} \cdot h_2|_x
  \]
  for all \(h_1,h_2\in \Gr[H]\), \(x\in A\).

  The bijection \(\Bisp\cong A\times\Gr\) is unique only up to a map
  of the form \((x,g)\mapsto (x,\psi(x)\cdot g)\) for some map
  \(\psi\colon A\to\Gr\).  This does not change the action~\(\pi\)
  and replaces the cocycle~\(\varphi\) by
  \begin{equation}
    \label{eq:cocycle_conjugation}
    \varphi^\psi(h,x) \defeq
    \psi(\pi_h(x))^{-1}\cdot \varphi(h,x)\cdot \psi(x)
  \end{equation}
  because \((\pi_h(x),h|_x\cdot \psi(x)\cdot g) =
  (\pi_h(x),\psi(\pi_h(x)) \cdot \varphi^\psi(h,x)\cdot g)\).  The map
  \(\psi\mapsto \varphi^\psi\) defines a right
  action of the group of maps \(\psi\colon A\to\Gr\) on the set of
  \(1\)\nb-cocycles \(\varphi\colon \Gr[H]\times A\to\Gr\).  The
  \(\Gr[H]\)\nb-space~\(A\) and the orbit of~\(\varphi\) under this
  action are uniquely determined by the isomorphism class of the
  groupoid correspondence \(\Gr[H]\leftarrow \Gr\), and any such pair
  \((A,\varphi)\) comes from a groupoid correspondence
  \(\Gr[H]\leftarrow \Gr\).

  Since~\(\Gr[H]^0\) is the one-point set as well, a groupoid
  correspondence is proper if and only if the set \(A\cong \Bisp/\Gr\)
  is finite, and tight if and only if \(\abs{A}=1\) or, equivalently,
  the right \(\Gr\)\nb-action on~\(\Bisp\) is transitive.  In the
  tight case, the construction above identifies \(\Bisp\cong \Gr\) as
  a right \(\Gr\)\nb-space by picking a base point in~\(\Bisp\).  The
  left action must be of the form \(h\cdot g = \varphi(h) g\) for a
  group homomorphism \(\varphi\colon \Gr[H]\to \Gr\) in order to
  commute with the right \(\Gr\)\nb-action.  If we pick another base
  point \(g\in\Gr\cong \Bisp\), then the homomorphism~\(\varphi\) is
  replaced by \({\Ad_g^{-1}}\circ \varphi\)
  by~\eqref{eq:cocycle_conjugation}.  Thus isomorphism classes of
  tight groupoid correspondences \(\Gr[H]\leftarrow \Gr\) are
  canonically in bijection with equivalence classes of group
  homomorphisms \(\Gr[H]\to \Gr\), where we consider two homomorphisms
  equivalent if they differ by an inner automorphism of~\(\Gr\).

  An injective group homomorphism \(\varphi\colon \Gr\to\Gr[H]\) gives
  a tight groupoid correspondence \(\Gr\leftarrow \Gr[H]\).  But it
  also gives a groupoid correspondence \(\Gr[H]\leftarrow \Gr\) by
  taking~\(\Gr[H]\) with the left \(\Gr[H]\)\nb-action by translation
  and the right \(\Gr\)\nb-action \(h\cdot g \defeq h\varphi(g)\).
  Such groupoid correspondences for \(\Gr=\Gr[H]\) are implicitly used
  by Stammeier~\cite{Stammeier:Irreversible}.  If \(\Gr\)
  and~\(\Gr[H]\) are Abelian discrete groups, an injective group
  homomorphism \(\varphi\colon \Gr\to\Gr[H]\) is equivalent to a
  surjective group homomorphism \(\hat\varphi\colon
  \hat{\Gr[H]}\to\hat{\Gr}\).  These are used by Cuntz and
  Vershik~\cite{Cuntz-Vershik:Endomorphisms}.
\end{example}

The analysis in \longref{Example}{exa:self-similar_groups} shows that a
proper groupoid correspondence \(\Gr\leftarrow \Gr\)
for a group~\(\Gr\)
is the same as a ``covering permutational bimodule'' over~\(\Gr\)
in the notation of \cite{Nekrashevych:Cstar_selfsimilar}*{Section~2}.
These covering permutational bimodules are another way to describe
self-similarities of groups.  It is customary, however, to assume a
certain faithfulness property for self-similarities (see
\cite{Nekrashevych:Cstar_selfsimilar}*{Definition~2.1}).  To formulate
it, let~\(A^*\)
be the set of finite words over~\(A\).
The \(1\)\nb-cocycle
allows to extend the action of~\(\Gr\)
on~\(A\) to an action on~\(A^*\) by the recursive formula
\begin{equation}
  \label{selfsimilaraction}
  g\cdot (x w) = g(x) (g|_x\cdot w)
\end{equation}
for \(g\in \Gr\), \(x\in A\), \(w\in A^*\).  The triple consisting
of the group~\(\Gr\), the \(\Gr\)\nb-set~\(A\), and the
\(1\)\nb-cocycle \(\Gr\times A\to \Gr\) is called a
\emph{self-similar group} if~\(A\) is finite and the action
of~\(\Gr\) on~\(A^*\) defined above is faithful.  The latter
condition ensures that a self-similarity is sufficiently nontrivial.
We shall not use it in this article.

The relationship to self-similar groups suggests to view a proper
correspondence \(A\colon \Gr\leftarrow \Gr\)
for a groupoid~\(\Gr\)
as a self-similarity of~\(\Gr\).
What does this mean for a transformation groupoid \(\Gamma\ltimes V\),
where~\(\Gamma\) is a group and~\(V\) is a left \(\Gamma\)\nb-set?

\begin{proposition}
  \label{pro:self-similar_graph}
  Let~\(\Gamma\)
  be a discrete group and let~\(V\)
  be a left \(\Gamma\)\nb-set.
  Let \(\Gr \defeq \Gamma\ltimes V\)
  be the transformation groupoid.  Let~\(E\)
  be a left\/ \(\Gamma\)\nb-set
  with a \(1\)\nb-cocycle
  \(\varphi\colon \Gamma\times E\to \Gamma\),
  \((h,e)\mapsto h|_e\),
  that is, \((g h)|_e = g|_{h\cdot e} \cdot h|_e\)
  for all \(g,h\in \Gamma\),
  \(e\in E\)
  and with maps \(\rg,\s\colon E\rightrightarrows V\) that satisfy
  \begin{equation}
    \label{eq:s_r_self-similar_trafo}
    \s(g\cdot e) = (g|_e)\cdot \s(e)\quad\text{and}\quad
    \rg(g\cdot e) = g\cdot \rg(e)
  \end{equation}
  for all \(g\in \Gamma\), \(e\in E\).  Then
  \(\Bisp\defeq E\times \Gamma\) with the discrete topology, the
  anchor maps \(\rg,\s\colon \Bisp\rightrightarrows V\),
  \(\rg(e,g) = \rg(e)\), \(\s(e,g) = g^{-1}\cdot \s(e)\), with the
  obvious right \(\Gamma\)\nb-action,
  \((e,g)\cdot g_2 = (e, g\cdot g_2)\), and the left
  \(\Gamma\)\nb-action \(h\cdot (e,g) = (h\cdot e, h|_e\cdot g)\) is
  a groupoid correspondence \(\Gr \leftarrow \Gr\).  Any groupoid
  correspondence \(\Gr\leftarrow\Gr\) is isomorphic to one of this
  form, where \((E,\rg,\s)\) is unique up to isomorphism,
  and~\(\varphi\) is unique up to the action of the group of maps
  \(\psi\colon E\to \Gamma\) by
  \[
    \varphi^\psi (h,e) \defeq \psi(\pi_h(e))^{-1}\cdot
    \varphi(h,e)\cdot \psi(e).
  \]
  The correspondence~\(\Bisp\) is proper if and only if the map
  \(\rg\colon E\to V\) is finite-to-one, and tight if and only if
  the map \(\rg\colon E\to V\) is bijective.
\end{proposition}

\begin{proof}
  An action of \(\Gamma\ltimes V\)
  on~\(\Bisp\)
  is equivalent to a pair consisting of a \(\Gamma\)\nb-action
  on~\(\Bisp\)
  and a \(\Gamma\)\nb-equivariant
  map \(\Bisp\to V\).
  Thus a groupoid correspondence \(\Bisp\colon \Gr\leftarrow \Gr\)
  is a space with commuting left and right actions of~\(\Gamma\)
  and with anchor maps \(\rg,\s\colon \Bisp\rightrightarrows V\),
  with some extra properties.  Since~\(V\)
  is discrete and \(\s\colon \Bisp\to V\)
  is a local homeomorphism, \(\Bisp\)
  must be discrete.  Then the right \(\Gamma\ltimes V\)\nb-action
  is basic if and only if the right \(\Gamma\)\nb-action
  is free.  Choose a fundamental domain \(E\subseteq\Bisp\)
  for it.  Then the map \(E\times \Gamma\to\Bisp\),
  \((e,g)\mapsto e\cdot g\),
  is a homeomorphism.  We use it to identify~\(\Bisp\)
  with~\(E\times \Gamma\).
  Then the right \(\Gamma\)\nb-action
  becomes \((e,g_1)\cdot g_2 = (e,g_1\cdot g_2)\).
  The anchor maps satisfy
  \(\s(e,g) = \s(e\cdot g) = g^{-1}\cdot \s(e)\)
  and \(\rg(e,g) = \rg(e\cdot g) = \rg(e)\)
  for all \(e\in E\),
  \(g\in \Gamma\)
  because~\(\s\)
  is equivariant and~\(\rg\)
  is invariant for the right \(\Gamma\)\nb-action.
  Thus \(\s\)
  and~\(\rg\)
  are determined by their restrictions to~\(E\),
  which we also denote by \(\s\) and~\(\rg\).

  For \(e\in E\),
  we may write \(h\cdot (e,1) = (h\cdot e, h|_e)\)
  with \(h\cdot e\in E\),
  \(h|_e\in \Gamma\).
  As in \longref{Example}{exa:self-similar_groups}, the left action
  of~\(\Gamma\)
  on~\(\Bisp\)
  must be of the form \(h\cdot (e,g) = (h \cdot e,h|_e \cdot g)\)
  because it commutes with the right \(\Gamma\)\nb-action;
  and this gives a left \(\Gamma\)\nb-action
  on~\(E\times \Gamma\)
  if and only if \((h_1\cdot h_2)\cdot e = h_1\cdot (h_2\cdot e)\)
  and \((h_1 h_2)|_e = h_1|_{h_2\cdot e} \cdot h_2|_e\)
  for all \(h_1,h_2\in \Gamma\),
  \(e\in E\).
  That is, \(E\)
  is a left \(\Gamma\)\nb-space
  and the map \(\varphi\colon \Gamma\times E\to \Gamma\),
  \((h,e)\mapsto h|_e\), is a \(1\)\nb-cocycle.

  The map~\(\rg\)
  is equivariant for the left \(\Gamma\)\nb-action.  Hence
  \[
    \rg(h\cdot e)
    = \rg(h\cdot e,h|_e\cdot g)
    = \rg(h\cdot (e,g))
    = h\cdot \rg(e,g)
    = h\cdot \rg(e)
  \]
  for all \(g,h\in \Gamma\),
  \(e\in E\).
  The map \(\s\)
  is invariant for the left \(\Gamma\)\nb-actions.  Hence
  \[
    g^{-1}\cdot (h|_e)^{-1} \cdot \s(h\cdot e)
    = \s(h \cdot e,h|_e \cdot g)
    = \s(h \cdot (e,g))
    = \s(e,g)
    = g^{-1} \cdot \s(e)
  \]
  for all \(h,g\in \Gamma\),
  \(e\in E\).
  Equivalently, \(\s(h\cdot e) = h|_e \cdot \s(e)\)
  for all \(h\in \Gamma\),
  \(e\in E\).
  So \(\rg\)
  and~\(\s\)
  satisfy the two conditions in~\eqref{eq:s_r_self-similar_trafo}.

  By now, we have seen that any correspondence is of the asserted
  form.  The only choice in the construction is that of a fundamental
  domain for the free \(\Gamma\)\nb-action
  on~\(\Bisp\).
  Two such choices differ by right multiplication with a map
  \(E\to \Gamma\).
  Hence isomorphisms of groupoid correspondences are in bijection with
  pairs \((f,\psi)\),
  where~\(f\)
  is a bijection \(f\colon E\congto E'\)
  that intertwines the range and source maps and
  \(\psi\colon E\to \Gamma\)
  is such that \(\varphi'(h,f(e)) = \varphi^\psi(h,e)\)
  for all \(h\in \Gamma\),
  \(e\in E\).
  Since \(E\cong \Bisp/\Gr\)
  and both \(E\)
  and~\(V\)
  are discrete, a groupoid correspondence is tight or proper,
  respectively, if and only if the corresponding map
  \(\rg\colon E\to V\) is bijective or finite-to-one.
\end{proof}

It is possible, though not recommended, to view the space~\(E\) in
\longref{Proposition}{pro:self-similar_graph} as a directed graph
with vertex set~\(V\) and range and source maps~\(\rg,\s\).  The
group~\(\Gamma\) acts both on the vertices and the edges in this
graph, and~\(\rg\) is equivariant.  But the map \(\s\colon E\to V\)
is not \(\Gamma\)\nb-equivariant, so~\(\Gamma\) does not act by
graph automorphisms.  The extra condition
\(\s(g\cdot e) = g\cdot \s(e)\), which says that the
\(\Gamma\)\nb-action preserves the graph structure, is equivalent to
\(g|_e\cdot \s(e) = g\cdot \s(e)\) for all \(g\in \Gamma\),
\(e\in E\).  Exel and Pardo~\cite{Exel-Pardo:Self-similar} define a
self-similar graph as a graph with an action of~\(\Gamma\) by graph
automorphisms, such that \(g|_e\cdot x = g\cdot x\) for all
\(g\in \Gamma\), \(e\in E\), \(x\in V\) (see
\cite{Exel-Pardo:Self-similar}*{Equation (2.3.1)}).  Thus a
self-similar graph as in~\cite{Exel-Pardo:Self-similar} gives a
groupoid correspondence on the transformation
groupoid~\(\Gamma\ltimes V\).  The converse is not true, however.
The assumption \(\s(g\cdot e) = (g|_e)\cdot \s(e)\) is also found in
\cite{Laca-Raeburn-Ramagge-Whittaker:Equilibrium_self-similar_groupoid}*{Lemma~3.4}
in the context of self-similar actions of groupoids.

We suggest
that the right common generalisation of graph \(\Cst\)\nb-algebras
and Nekrashevych's \(\Cst\)\nb-algebras is a proper groupoid
correspondence \(\Gr\leftarrow\Gr\) for a discrete groupoid~\(\Gr\).
The idea of~\cite{Exel-Pardo:Self-similar} to look at self-similar
graphs restricts~\(\Gr\) to be a transformation
groupoid~\(\Gamma\ltimes V\) for a group action on a discrete set
--~the vertices of the graph~-- and it leads to the unnecessary
condition \(\s(g\cdot e) = g\cdot \s(e)\) on the source map
\(\s\colon E\to V\) in order for~\(\Gamma\) to act by graph
automorphisms.  We now show that the setting
in~\cite{Laca-Raeburn-Ramagge-Whittaker:Equilibrium_self-similar_groupoid}
is almost equivalent to that of a groupoid correspondence on a
discrete groupoid, except for an extra faithfulness condition
in~\cite{Laca-Raeburn-Ramagge-Whittaker:Equilibrium_self-similar_groupoid},
which is analogous to the faithfulness condition in the definition
of a self-similar group.

\begin{example}
  \label{exa:self-similar_groupoid}
  Let~\(\Gr\) be a (discrete) groupoid with object set~\(V\).  Let
  \(\Bisp\colon \Gr\leftarrow\Gr\) be a groupoid correspondence.  As
  above, the space~\(\Bisp\) is discrete and there is a fundamental
  domain~\(E\) with a map \(\s\colon E\to V\) for the right
  \(\Gr\)\nb-action on it such that
  \(\Bisp \cong E\times_{\s,V,\rg}\Gr\) as a right \(\Gr\)\nb-set.
  The left action of~\(\Gr\) induces an action on
  \(E\cong \Bisp/\Gr\).  Let \(\rg\colon E\to V\) denote its anchor
  map.  There is a map
  \(\varphi\colon \Gr\times_{\s,V,\rg} E\to\Gr\),
  \((g,e)\mapsto g|_e\), such that \(\s(g|_e) = \s(e)\) and
  \(g\cdot (e,h) = (g\cdot e, g|_e\cdot h)\) for all
  \(g,h \in \Gr\), \(e\in E\) with \(\s(g) = \rg(e)\),
  \(\s(e) = \rg(h)\).  The maps above must satisfy
  \(\s(g|_e) = \s(e)\), \(\rg(g\cdot e) = \rg(g)\),
  \(\s(g\cdot e) = \rg(g|_e)\), and the cocycle condition
  \((g\cdot h)|_e = g|_{h\cdot e} \cdot h|_e\) for all
  \(g,h\in\Gr\), \(e\in E\) with \(\s(g) =\rg(h)\),
  \(\s(h) = \rg(e)\) in order for
  \(g\cdot (e,h) = (g\cdot e, g|_e\cdot h)\) to define a groupoid
  action on~\(E\times_{\s,V,\rg}\Gr\).  Conversely, if we are given
  a \(\Gr\)\nb-set~\(E\) with a cocycle satisfying these conditions,
  then it comes from a unique groupoid correspondence.  So the only
  difference between a groupoid correspondence on~\(\Gr\) and a
  self-similar action of~\(\Gr\) as in
  \cite{Laca-Raeburn-Ramagge-Whittaker:Equilibrium_self-similar_groupoid}*{Definition~3.3}
  is the assumption
  in~\cite{Laca-Raeburn-Ramagge-Whittaker:Equilibrium_self-similar_groupoid}
  that the induced action of~\(\Gr\) on the space of finite paths is
  faithful.
\end{example}

\section{Composition of groupoid correspondences}
\label{sec:composition}

We are ready to define the composition of groupoid correspondences
and see that these form a bicategory.  Let \(\Gr[H]\), \(\Gr\)
and~\(\Gr[K]\) be étale groupoids and let
\(\Bisp\colon \Gr[H]\leftarrow \Gr\) and
\(\Bisp[Y]\colon \Gr\leftarrow \Gr[K]\) be groupoid correspondences.
Let
\[
  \Bisp\times_{\Gr^0}\Bisp[Y] \defeq
  \Bisp\times_{\s,\Gr^0,\rg}\Bisp[Y]
  \defeq \setgiven{(x,y)\in \Bisp\times \Bisp[Y]}{\s(x)=\rg(y)}.
\]
Let~\(\Gr\) act on \(\Bisp\times_{\Gr^0}\Bisp[Y]\) by the
\emph{diagonal action}
\[
  g\cdot (x,y)\defeq (x\cdot g^{-1},g\cdot y)
\]
for \(x\in\Bisp\), \(y\in\Bisp[Y]\) and \(g\in \Gr\) with
\(\s(g)=\rg(y) = \s(x)\).  Let \(\Bisp\Grcomp_{\Gr}\Bisp[Y]\) be the
orbit space of this action.  The image of
\((x,y)\in \Bisp\times_{\Gr^0}\Bisp[Y]\) in
\(\Bisp\Grcomp_{\Gr}\Bisp[Y]\) is usually denoted by \([x,y]\).

The maps \(\rg(x,y)\defeq \rg(x)\) and \(\s(x,y)\defeq \s(y)\) on
\(\Bisp\times_{\Gr^0}\Bisp[Y]\) are invariant for this action and
thus induce maps
\(\rg\colon \Bisp\Grcomp_{\Gr}\Bisp[Y]\to \Gr[H]^0\) and
\(\s\colon \Bisp\Grcomp_{\Gr}\Bisp[Y]\to \Gr[K]^0\).  These are the
anchor maps for the commuting actions of~\(\Gr[H]\) on the left
and~\(\Gr[K]\) on the right, which we define by
\[
  h\cdot[x,y]\defeq [h \cdot x,y],\qquad
  [x,y]\cdot k\defeq [x,y\cdot k]
\]
for all \(h\in \Gr[H]\), \(x\in\Bisp\), \(y\in\Bisp[Y]\),
\(k\in\Gr[K]\) with \(\s(h) = \rg(x)\), \(\s(x)=\rg(y)\), and
\(\s(y) = \rg(k)\).  This is well defined because
\([h\cdot x\cdot g^{-1},g\cdot y] = [h\cdot x,y]\) and
\([x\cdot g^{-1},g\cdot y\cdot k] = [x,y\cdot k]\) for \(g\in\Gr\)
with \(\s(g) = \s(x) = \rg(y)\).

We are going to prove that \(\Bisp\Grcomp_{\Gr}\Bisp[Y]\) with these
two actions is again a groupoid correspondence.  The following
lemmas are needed for this.  In some of the statements, we use the
construction of \(\Bisp\Grcomp_{\Gr} \Bisp[Y]\) also when
\(\Bisp[Y]\) is merely a left \(\Gr\)\nb-space, but there is no
groupoid~\(\Gr[K]\) that acts on~\(\Bisp[Y]\) on the right.  Then
\(\Bisp\Grcomp_{\Gr} \Bisp[Y]\) is still a left \(\Gr[H]\)\nb-space.

\begin{lemma}
  \label{lem:pullback_proper}
  The pullback of a proper map is also a proper map.
\end{lemma}

\begin{proof}
  We form the pullback of two continuous maps $\alpha\colon A \to C$
  and $\beta\colon B \to C$ as in~\eqref{pullback}.  We
  assume~\(\beta\) to be proper, that is, stably closed, and want to
  prove the same for the pullback map
  \(\pi_1\colon A\times_C B \to A\).  Let~$X$ be another space.
  Then the map map
  $\id_{X \times A}\times \beta \colon X \times A \times B \to X
  \times A \times C$ is closed because~\(\beta\) is proper.  The
  space \(X\times A\) is homeomorphic to a subspace of
  \(X\times A \times C\) by the embedding
  \(j\colon X\times A \to X\times A \times C\),
  \((x,a)\mapsto (x,a,\alpha(a))\); this is an embedding because the
  projection to the first two coordinates is a one-sided inverse.
  The space \(X\times A \times_{\alpha,C,\beta} B\) of all
  \((x,a,b)\) with \(\alpha(a) = \beta(b)\) is the preimage of
  \(j(X\times A)\) under \(\id_{X \times A}\times \beta\), and
  \(\id_X \times \pi_1\colon X\times A \times_{\alpha,C,\beta} B \to
  X\times A\) is the restriction of \(\id_{X \times A}\times \beta\)
  to this preimage, composed with~\(j^{-1}\).  A subset of
  \(X\times A \times_{\alpha,C,\beta} B\) is closed if and only if
  it is of the form \(D\cap (X\times A \times_{\alpha,C,\beta} B)\)
  for a closed subset~\(D\) of \(X\times A \times B\).  Then
  \(\id_X \times \pi_1\) maps it to the \(j\)\nb-preimage of
  \((\id_{X \times A}\times \beta)(D)\), which is closed because
  \(\id_{X \times A}\times \beta\) is closed and~\(j\) is
  continuous.  Thus \(\id_X \times \pi_1\) is closed.
\end{proof}

\begin{lemma}
  \label{compositehausdorffandproper}
  Let~$A$ be a right $\Gr$\nb-space and~$B$ a left \(\Gr\)\nb-space.
  If the \(\Gr\)\nb-action on~\(A\) is proper and~\(B\) is
  Hausdorff, then the diagonal $\Gr$\nb-action on
  $A \times_{\s, \Gr^0, \rg} B$ defined by
  \(g\cdot (a,b) \defeq (a\cdot g^{-1},g\cdot b)\) is proper.
\end{lemma}

\begin{proof}
  Since~\(B\) is Hausdorff, the diagonal inclusion
  \[
    \Delta\colon B\to B\times B,\qquad
    b\mapsto (b,b),
  \]
  is a closed map.  Since~\(\Delta\) is also injective, it is even a
  proper map.  By \longref{Lemma}{lem:pullback_proper}, the pullback
  of~\(\Delta\) along any map into \(B\times B\) is again proper.
  It is useful to generalise this result a bit.  Consider maps
  \(\alpha\colon A\to C\) and maps \(f\colon B_1\to B_2\) and
  \(\beta\colon B_2 \to C\).  In the diagram
  \[
    \begin{tikzcd}
      A \times_{C} B_1 \arrow[r, "\pi_2"] \arrow[d, "\id_A \times f"']
      \arrow[dr, phantom, "\scalebox{1.5}{$\lrcorner$}" , very near
      start, color=black]
      & B_1 \arrow[d, "f"] \\
      A \times_{C} B_2 \arrow[r, "\pi_2"] \arrow[d, "\pi_1"']
      \arrow[dr, phantom, "\scalebox{1.5}{$\lrcorner$}" , very near
      start, color=black]
      & B_2 \arrow[d, "\beta"] \\
      A \arrow[r, "\alpha"'] & C
    \end{tikzcd}
  \]
  the lower square and the whole rectangle are pullbacks, and this
  implies that the top square is a pullback square as well.
  Therefore, the map \(\id_A\times f\) is a pullback of~\(f\) and
  inherits the property of being proper from~\(f\).

  We now form this kind of pullback of~\(\Delta\) along the maps
  \(B\times B\to \Gr^0\times \Gr^0\),
  \((b_1,b_2) \mapsto (\rg(b_1),\rg(b_2))\), and
  \(\Gr\times_{\s,\Gr^0,\s} A\to \Gr^0\times \Gr^0\),
  \((g,a) \mapsto (\s(g),\s(a)) = (\s(g),\s(g))\).  This gives a map
  from the space of triples \((g,a,b) \in \Gr\times A\times B\) with
  \(\s(g) = \s(a) = \rg(b)\) to the space of quadruples
  \((g,a,b_1,b_2)\) with \(\s(g) = \s(a) = \rg(b_1)=\rg(b_2)\).  The
  formula \((g,a,b_1,b_2)\mapsto (g,a,g\cdot b_1,b_2)\) defines a
  homeomorphism from the target of this map to the space of
  quadruples \((g,a,b_1,b_2)\in \Gr\times A\times B\times B\) with
  \(\s(g) = \s(a)=\rg(b_2)\) and \(\rg(g) = \rg(b_1)\); the inverse
  is defined by \((g,a,b_1,b_2)\mapsto (g,a,g^{-1}\cdot b_1,b_2)\).

  Since the \(\Gr\)\nb-action on~\(A\) is proper, the following map
  is proper:
  \[
    \varphi\colon \Gr\times_{\s,\Gr^0,\s} A \to A\times A,\qquad
    (g,a) \mapsto (a\cdot g^{-1},a).
  \]
  We map \(A\times A\to \Gr^0 \times \Gr^0\),
  \((a_1,a_2)\mapsto (\s(a_1),\s(a_2))\), and
  \(B\times B \to (\Gr^0,\Gr^0)\),
  \((b_1,b_2) \mapsto (\rg(b_1),\rg(b_2))\).  Then the pullback of
  the proper map~\(\varphi\) along these maps becomes the map
  \((g,a,b_1,b_2) \mapsto (a\cdot g^{-1},a,b_1,b_2)\) from the space
  of all quadruples \((g,a,b_1,b_2)\in \Gr\times A\times B\times B\)
  with \(\s(g) = \s(a) = \rg(b_2)\) and
  \(\s(a\cdot g^{-1}) = \rg(b_1)\) to the space of all quadruples
  \((a_1,a_2,b_1,b_2)\in A\times A\times B\times B\) with
  \(\s(a_1) = \rg(b_1)\) and \(\s(a_2) = \rg(b_2)\).  This map is
  again proper as a pullback of a proper map.  Now use
  \(\s(a\cdot g^{-1}) = \rg(g)\) to identify the domain of this map
  with the codomain of the proper map that we constructed above
  from~\(\Delta\).  Composing the two proper maps above gives the
  map \((g,a,b) \mapsto (a\cdot g^{-1},a,g\cdot b,b)\) from the
  space of triples \((g,a,b) \in \Gr\times A\times B\) with
  \(\s(g) = \s(a) = \rg(b)\) to the space of quadruples
  \((a_1,a_2,b_1,b_2)\) with \(\s(a_1)= \rg(b_1)\) and
  \(\s(a_2) =\rg(b_2)\).  After exchanging the order of \(a_2\)
  and~\(b_1\), this becomes the map that witnesses that the diagonal
  \(\Gr\)\nb-action on \(A\times_{\s,\Gr^0,\rg} B\) is proper.
\end{proof}

\begin{example}
  By \longref{Proposition}{pro:basic_actions_free_proper}, the trivial
  action of the trivial group on a space~\(B\) is proper if and only
  if~\(B\) is Hausdorff.  This example shows that
  \longref{Lemma}{compositehausdorffandproper} becomes false if we do not
  assume~\(B\) to be Hausdorff.
\end{example}

\begin{lemma}
  \label{lem:exchange_orbit_fibreprod}
  We have $(\Bisp \times_{\s, \Gr^0, \rg }
  \Bisp[Y])/\Gr[K] \cong  \Bisp \times_{\s, \Gr^0,
    \rg_{\Bisp[Y]_*}} (\Bisp[Y]/\Gr[K])$.
\end{lemma}

\begin{proof}
  The map
  $\Qu_{\Bisp[Y]}\colon \Bisp[Y] \twoheadrightarrow \Bisp[Y]/\Gr[K]$
  is a local homeomorphism by \longref{Lemma}{lem:basic_orbit_lh}.  Then
  so is any pullback of it by \longref{Lemma}{lem:pullback_local_homeo}.
  As in the proof of \longref{Lemma}{compositehausdorffandproper}, this
  applies to the map
  \[
    \id_{\Bisp} \times_{\Gr^0} \Qu_{\Bisp[Y]}\colon
    \Bisp \times_{\s, \Gr^0, \rg } \Bisp[Y] \to
    \Bisp \times_{\s, \Gr^0, \rg_{\Bisp[Y]_*}} (\Bisp[Y]/\Gr[K]),
  \]
  which is equivalent to the pullback along the maps
  \(\rg_{\Bisp[Y]_*}\colon \Bisp[Y]/\Gr[K] \to \Gr^0\) and
  \(\s\colon \Bisp \to \Gr^0\) because
  \(\rg_{\Bisp[Y]_*}\circ \Qu_{\Bisp[Y]} = \rg\).  The quotient map
  \(\Qu\colon\Bisp \times_{\s, \Gr^0, \rg } \Bisp[Y] \to (\Bisp
  \times_{\s, \Gr^0, \rg_{\Bisp[Y]_*}} \Bisp[Y])/\Gr[K]\) is a local
  homeomorphism as well by \longref{Lemma}{compositehausdorffandproper}.
  By the universal property of the orbit space, we get a commuting
  diagram of continuous maps
  \[
    \begin{tikzcd}[column sep = huge]
      \Bisp \times_{\s, \Gr^0, \rg } \Bisp[Y]
      \dar["\Qu"'] \ar[dr,"\id_{\Bisp} \times_{\Gr^0}
      \Qu_{\Bisp[Y]}"]\\
      (\Bisp \times_{\s, \Gr^0, \rg }
      \Bisp[Y])/\Gr[K] \rar["(\id_{\Bisp} \times_{\Gr^0}
      \Qu_{\Bisp[Y]})_*"']&
      \Bisp \times_{\s, \Gr^0, \rg_{\Bisp[Y]_*}}
      (\Bisp[Y]/\Gr[K]).
    \end{tikzcd}
  \]
  The map in the bottom is easily seen to be bijective.  Since both
  maps that go down are surjective and local homeomorphisms, it
  follows that the bottom map is a local homeomorphism as well.
  Being bijective, it is a homeomorphism.
\end{proof}

Since the map
$\pi_1'\colon (\Bisp \times_{\s, \Gr^0, \rg } \Bisp[Y])/\Gr[K] \cong
\Bisp \times_{\s, \Gr^0, \rg } (\Bisp[Y]/\Gr[K]) \to \Bisp$ is
\(\Gr\)\nb-equivariant, it induces a map
$(\pi_1')_*\colon (\Bisp \Grcomp_{\Gr} \Bisp[Y])/\Gr[K] \to
\Bisp/\Gr$ on the \(\Gr\)\nb-orbit spaces.

\begin{lemma}
  \label{lem:proper_on_orbits}
  If~\(\pi_1'\) is proper, then so is~$(\pi_1')_*$.
\end{lemma}

\begin{proof}
  Let~\(Z\) be a topological space.
  \longref{Lemma}{lem:exchange_orbit_fibreprod} implies
  \((\Bisp/\Gr)\times Z \cong (\Bisp\times Z)/\Gr\) and
  \((\Bisp \times_{\s, \Gr^0, \rg } \Bisp[Y])/\Gr[K] \times Z \cong
  (\Bisp \times_{\s, \Gr^0, \rg } \Bisp[Y] \times Z)/\Gr[K]\).
  Consider the following diagram:
  \[
    \begin{tikzcd}
      (\Bisp \times_{\s, \Gr^0, \rg }
      \Bisp[Y] \times Z)/\Gr[K] \arrow[r, "\pi_1' \times \id_Z"] \arrow[d, "\Qu"']
      & \Bisp \times Z \arrow[d, "\Qu_{\Bisp}"] \\
      (\Bisp \Grcomp_{\Gr} \Bisp[Y]\times Z)/\Gr[K]
      \arrow[r, "(\pi_1')_*\times \id_Z"'] & (\Bisp/\Gr)\times Z
    \end{tikzcd}
  \]
  Let~$A$ be a closed subset in
  $(\Bisp \Grcomp_{\Gr} \Bisp[Y]\times Z)/\Gr[K]$.  We abbreviate
  \(f\defeq \pi_1' \times \id_Z\) Since~$\Qu$ is continuous and~$f$
  is proper, $f(\Qu^{-1}(A))$ is closed.  Then
  $\Bisp\backslash f(\Qu^{-1}(A))$ is open in~$\Bisp\times Z$.  It
  consists of those $(x,z)$ whose \(\Gr\)\nb-orbit is disjoint
  from~$f(A)$.  The map~$\Qu_{\Bisp}$ is open by
  \longref{Lemma}{lem:basic_orbit_lh}.  So
  $\Qu_{\Bisp}(\Bisp\times Z\backslash f(\Qu^{-1}(A))) =
  (\Bisp/\Gr\times Z) \backslash \Qu_{\Bisp}(f(\Qu^{-1}(A)))$ is
  open in $\Bisp/\Gr\times Z$.  Thus
  $\Qu_{\Bisp}(f(\Qu^{-1}(A))) = (\pi_1')_*\times \id_Z(A)$ is
  closed.
\end{proof}

\begin{lemma}
  \label{lem:commute_two_orbits}
  There is a canonical homeomorphism
  $(\Bisp \Grcomp_{\Gr} \Bisp[Y])/\Gr[K] \cong \Bisp \Grcomp_{\Gr}
  (\Bisp[Y]/\Gr[K])$.
\end{lemma}

\begin{proof}
  Let
  $\pr_0 \colon (\Bisp \times_{\s, \Gr^0, \rg } \Bisp[Y])/\Gr[K] \to
  (\Bisp \times_{\s, \Gr^0, \rg } \Bisp[Y])/(\Gr \times \Gr[K])$ be
  the orbit space projection.  By definition,
  $(\Bisp \times_{\s, \Gr^0, \rg } \Bisp[Y])/(\Gr \times \Gr[K])
  \cong (\Bisp \Grcomp_{\Gr} \Bisp[Y])/\Gr[K]$.  Thus we may define
  $\pr_1 \colon (\Bisp \times_{\s, \Gr^0, \rg } \Bisp[Y])/\Gr[K] \to
  (\Bisp \Grcomp_{\Gr} \Bisp[Y])/\Gr[K]$.  Let
  $\pr_2\colon \Bisp \times_{\s, \Gr^0, \rg_{\Bisp[Y]_*}}
  (\Bisp[Y]/\Gr[K]) \to \Bisp \Grcomp_{\Gr} (\Bisp[Y]/\Gr[K])$ be the
  orbit space projection.  There is a commutative diagram
  \[
    \begin{tikzcd}
      (\Bisp \times_{\s, \Gr^0, \rg }
      \Bisp[Y])/\Gr[K] \arrow[r, "\cong"] \arrow[d,
      "\pr_1"']
      & \Bisp \times_{\s, \Gr^0, \rg_{\Bisp[Y]_*}}
      (\Bisp[Y]/\Gr[K]) \arrow[d, "\pr_2"] \\
      (\Bisp \Grcomp_{\Gr} \Bisp[Y])/\Gr[K]
      \arrow[r, "g"'] & \Bisp \Grcomp_{\Gr}
      (\Bisp[Y]/\Gr[K]).
    \end{tikzcd}
  \]
  The map~$g$ is clearly bijective, and the homeomorphism in the top
  row is shown in \longref{Lemma}{lem:exchange_orbit_fibreprod}.  Since
  $\pr_1$ and~$\pr_2$ are surjective local homeomorphisms by
  \longref{Lemma}{lem:basic_orbit_lh}, $g$ is a homeomorphism.
\end{proof}

\begin{proposition}
  \label{pro:groupoid_correspondences_composition}
  The actions of \(\Gr[H]\) and~\(\Gr[K]\) on
  \(\Bisp\Grcomp_{\Gr}\Bisp[Y]\) are well defined and continuous and
  turn this into a groupoid correspondence
  \(\Gr[H]\leftarrow \Gr[K]\).

  If both correspondences \(\Bisp\) and~\(\Bisp[Y]\) are proper or
  tight, then so is \(\Bisp\Grcomp_{\Gr}\Bisp[Y]\).
\end{proposition}

\begin{proof}
  To show that the actions are well defined, let $[x', y'] = [x, y]$
  be two representatives for the same element in
  \(\Bisp\Grcomp_{\Gr}\Bisp[Y]\).  Then $x' = x \cdot g^{-1}$ and
  $y' = g \cdot y$ for some $g \in \Gr$.  Then
  $[h \cdot x', y'] = [(h \cdot x) \cdot g^{-1}, g \cdot y] =
  [h\cdot x,y]$ and
  $[x', y' \cdot k] = [x \cdot g^{-1}, g \cdot (y \cdot k)] =
  [x,y\cdot k]$.

  The actions on \(\Bisp \Grcomp_{\Gr} \Bisp[Y]\) are continuous
  because they are continuous on
  \(\Bisp \times_{\s,\Gr^0,\rg} \Bisp[Y]\) and
  \(\Gr[H] \times_{\s,\Gr[H]^0,\rg} (\Bisp \Grcomp_{\Gr} \Bisp[Y])
  \cong (\Gr[H] \times_{\s,\Gr[H]^0,\rg} \Bisp
  \times_{\s,\Gr^0,\rg} \Bisp[Y])/\Gr\) by
  \longref{Lemma}{lem:commute_two_orbits}, and similarly
  for~\(\Gr[K]\).

  To show that \(\Bisp\Grcomp_{\Gr}\Bisp[Y]\) is a groupoid
  correspondence, we first check that the actions commute.  Indeed,
  $\s(h \cdot [x, y]) = \s(y) = \s([x, y])$,
  $\rg([x, y] \cdot k) = \rg(x) = \rg([x, y])$, and
  $(h \cdot [x, y]) \cdot k = [h\cdot x,y\cdot k] = h \cdot ([x, y]
  \cdot k)$.

  Next we show that
  $\s\colon \Bisp \Grcomp_{\Gr} \Bisp[Y] \to \Gr[K]^0$ is a local
  homeomorphism.  Since $\s_{\Bisp}\colon \Bisp \to \Gr^0$ is a
  local homeomorphism, \longref{Lemma}{lem:pullback_local_homeo} implies
  that
  $\pi_{\Bisp[Y]} \colon \Bisp \times_{\s, \Gr^0, \rg } \Bisp[Y] \to
  \Bisp[Y]$ is a local homeomorphism as well.  Since $\s_{\Bisp[Y]}$
  is a local homeomorphism, too, the composite map in the top of the
  following diagram is a local homeomorphism:
  \[
    \begin{tikzcd}
      \Bisp \times_{\s, \Gr^0, \rg } \Bisp[Y]
      \ar[d, "\Qu"']
      \rar["\pi_{\Bisp[Y]}"]& \Bisp[Y] \rar["\s_{\Bisp[Y]}"]&
      \Gr[K]^0\\
      \Bisp \Grcomp_{\Gr} \Bisp[Y] \ar[urr, "\s"']
    \end{tikzcd}
  \]
  The map
  $\s\colon \Bisp \Grcomp_{\Gr} \Bisp[Y] \to \Gr[K]^0$ is defined so
  as to make this diagram commute, and the vertical map is a
  surjective local homeomorphism by \longref{Lemma}{lem:basic_orbit_lh}.
  Then it follows that~\(\s\) is a local homeomorphism as well.

  Next, we show that the $\Gr[K]$\nb-action on
  $\Bisp \Grcomp_{\Gr} \Bisp[Y]$ is basic.  It is easy to check that
  this action is free.  Therefore, if
  \([x_1,y_1],[x_2,y_2]\in\Bisp \Grcomp_{\Gr} \Bisp[Y]\) are in the
  same \(\Gr[K]\)\nb-orbit, then there is a unique \(k\in\Gr[K]\)
  with \(\s(y_1) = \s[x_1,y_1] = \rg(k)\) and
  \([x_2,y_2] = [x_1,y_1]\cdot k\).  We must show that~\(k\) depends
  continuously on the pair \([x_1,y_1],[x_2,y_2]\) --~subject to the
  condition that they lie in the same \(\Gr[K]\)\nb-orbit.  First,
  \([x_1,y_1]\cdot k = [x_1,y_1\cdot k]\), and this is equal to
  \([x_2,y_2]\) if and only if there is \(g\in\Gr\) with
  \(\s(x_1) = \rg(g)\) and
  \((x_2,y_2) = (x_1\cdot g,g^{-1}\cdot y_1\cdot k)\).  Then
  \longref{Lemma}{def:innerprodofcorr} implies \(g = \braket{x_1}{x_2}\)
  and
  \[
    k = \braket{g^{-1} y_1}{y_2}
    = \braket[\big]{\braket{x_2}{x_1}\cdot y_1}{y_2}.
  \]
  Since the bracket maps for \(\Bisp\) and~\(\Bisp[Y]\) are
  continuous, it follows that~\(k\) depends continuously on
  \((x_1,y_1),(x_2,y_2)\in\Bisp \times_{\s,\Gr^0,\rg} \Bisp[Y]\).
  Since the orbit space projection from
  \(\Bisp \times_{\s,\Gr^0,\rg} \Bisp[Y]\) to
  $\Bisp \Grcomp_{\Gr} \Bisp[Y]$ is a local homeomorphism by
  \longref{Lemma}{lem:basic_orbit_lh}, it follows that~\(k\) still
  depends continuously on
  \([x_1,y_1],[x_2,y_2]\in\Bisp \Grcomp_{\Gr} \Bisp[Y]\).

  To prove that $(\Bisp \Grcomp_{\Gr} \Bisp[Y])/\Gr[K]$ is a
  groupoid correspondence, it only remains to show that the orbit
  space $(\Bisp \Grcomp_{\Gr} \Bisp[Y])/\Gr[K]$ is Hausdorff; then
  \longref{Proposition}{pro:basic_actions_free_proper} shows that
  the right \(\Gr[K]\)\nb-action on \(\Bisp \Grcomp_{\Gr} \Bisp[Y]\)
  is free and proper.  Since the right $\Gr[K]$\nb-action
  on~$\Bisp[Y]$ is free and proper, $\Bisp[Y]/\Gr[K]$ is Hausdorff
  by \longref{Proposition}{pro:basic_actions_free_proper}.  Then
  \longref{Lemma}{compositehausdorffandproper} shows that the
  diagonal \(\Gr\)\nb-action on
  $\Bisp \times_{\s, \Gr^0, \rg_{\Bisp[Y]_*}} (\Bisp[Y]/\Gr[K])$ is
  proper.  By \longref{Proposition}{pro:basic_actions_free_proper},
  its orbit space $\Bisp \Grcomp_{\Gr} (\Bisp[Y]/\Gr[K])$ is
  Hausdorff.  Then $(\Bisp \Grcomp_{\Gr} \Bisp[Y])/\Gr[K]$ is
  Hausdorff by \longref{Lemma}{lem:commute_two_orbits}.

  Now assume that both correspondences $\Bisp$ and~$\Bisp[Y]$ are
  proper.  That is, the maps
  $\rg_{\Bisp_*}\colon \Bisp/\Gr \to \Gr[H]^0$ and
  $\rg_{\Bisp[Y]_*}\colon \Bisp[Y]/\Gr[K] \to \Gr^0$ are proper.
  \longref{Lemma}{lem:pullback_proper} applied to the pullback diagram
  \[
    \begin{tikzcd}
      \Bisp \times_{\s, \Gr^0, \rg_{\Bisp[Y]_*}}
      (\Bisp[Y]/\Gr[K]) \arrow[r, "\pi_2"] \arrow[d, "\pi_1"']
      \arrow[dr, phantom, "\scalebox{1.5}{$\lrcorner$}" , very near
      start, color=black]
      & \Bisp[Y]/\Gr[K] \arrow[d, "\rg_{\Bisp[Y]_*}"] \\
      \Bisp \arrow[r, "\s_{\Bisp}"'] & \Gr^0
    \end{tikzcd}
  \]
  shows that~$\pi_1$ is proper.  Thus the map
  $\pi_1'\colon (\Bisp \times_{\s, \Gr^0, \rg } \Bisp[Y])/\Gr[K] \to
  \Bisp$ defined through the homeomorphism in
  \longref{Lemma}{lem:exchange_orbit_fibreprod} is also proper.  Then the
  map~$(\pi_1')_*$ in \longref{Lemma}{lem:proper_on_orbits} is proper.  Then
  $\rg_{\Bisp_*} \circ (\pi_1')_*$ is proper, and this map is equal
  to \(\rg_{{\Bisp \Grcomp_{\Gr} \Bisp[Y]}_*}\).

  Finally, suppose that both correspondences
  $\Bisp\colon \Gr[H] \leftarrow \Gr$ and
  $\Bisp[Y]\colon \Gr \leftarrow \Gr[K]$ are tight.  Then we follow
  the argument in the proper case and observe instead that each of
  the maps \(\rg_{\Bisp_*}\), \(\pi_1\), \(\pi_1'\)
  and~$(\pi_1')_*$, \(\rg_{\Bisp[Y]_*}\) is a homeomorphism, and
  then so is the composite map
  $\rg_{{\Bisp \Grcomp_{\Gr} \Bisp[Y]}_*} = \rg_{\Bisp_*} \circ
  (\pi_1')_*$.
\end{proof}

\section{The bicategory of groupoid correspondences}
\label{sec:bicategory}

In this section, we define the bicategory~\(\Grcat\) of groupoid
correspondences.  Its objects are étale, locally compact groupoids
with Hausdorff object space, which we continue to call just
``groupoids'' (see \longref{Definition}{def:groupoid}).  Let \(\Gr[H]\)
and~\(\Gr\) be two such groupoids.  An arrow
\(\Gr[H] \leftarrow \Gr\) or \(\Gr \to \Gr[H]\) is a groupoid
correspondence \(\Bisp\colon \Gr[H] \leftarrow \Gr\) as in
\longref{Definition}{def:groupoid_corr}; beware that the source of such a
correspondence is on the right, as in our notation for the anchor
maps.  Let \(\Bisp,\Bisp[Y]\colon \Gr[H] \leftleftarrows \Gr\) be
two such groupoid correspondences.  A \(2\)\nb-arrow
\(\Bisp \Rightarrow \Bisp[Y]\) is an injective,
\(\Gr[H],\Gr\)-equivariant, continuous map
\(\alpha\colon \Bisp \to \Bisp[Y]\).  The following lemma shows that
such a \(2\)\nb-arrow is automatically open or, equivalently, a
homeomorphism onto an open subset of~\(\Bisp[Y]\):

\begin{lemma}
  \label{lem:2-arrow_local_homeo}
  Let \(\Bisp,\Bisp[Y]\colon \Gr[H] \leftleftarrows \Gr\) be
  groupoid correspondences.  Any \(\Gr[H],\Gr\)-equivariant,
  continuous map \(\alpha\colon \Bisp \to \Bisp[Y]\) is a local
  homeomorphism.  Therefore, a \(2\)\nb-arrow
  \(\Bisp \Rightarrow \Bisp[Y]\) is a homeomorphism from~\(\Bisp\)
  onto an open subset of~\(\Bisp[Y]\).  It is a homeomorphism
  onto~\(\Bisp[Y]\) if it is also surjective.
\end{lemma}

\begin{proof}
  Let \(x\in \Bisp\).  By assumption, both source maps are local
  homeomorphisms and \(s_{\Bisp[Y]} \circ \alpha = s_{\Bisp}\).  Let
  \(U_{\Bisp[Y]}\subseteq \Bisp[Y]\) be an open neighbourhood
  of~\(\alpha(x)\) on which~\(s_{\Bisp[Y]}\) is injective.
  Since~\(\alpha\) is continuous, there is an open neighbourhood
  \(U_{\Bisp} \subseteq \Bisp\) with
  \(\alpha(U_{\Bisp}) \subseteq U_{\Bisp[Y]}\).
  Shrinking~\(U_{\Bisp}\) further if necessary, we may arrange that
  \(s_{\Bisp}|_{U_{\Bisp}}\) is injective.  Then
  \(\alpha|_{U_{\Bisp}} \colon U_{\Bisp} \to U_{\Bisp[Y]}\) is equal
  to the map \(s_{\Bisp[Y]}|_{U_{\Bisp[Y]}}^{-1} \circ s_{\Bisp}\),
  and this is a homeomorphism from~\(U_{\Bisp}\) onto an open subset
  of~\(U_{\Bisp[Y]}\).  This implies that~\(\alpha\) is a local
  homeomorphism.  An injective local homeomorphism must be a
  homeomorphism onto an open subset of its codomain.
\end{proof}

The composition of \(2\)\nb-arrows is the obvious composition of
maps.  This is clearly associative, and the
identity maps on groupoid correspondences are units for this
composition.  Thus there is a category \(\Grcat(\Gr,\Gr[H])\)
with arrows \(\Gr[H]\leftarrow\Gr\) as objects and the
\(2\)\nb-arrows between as arrows.

\begin{remark}
  We would still get a bicategory in the same way if we allow all
  \(\Gr[H],\Gr\)-equivariant, continuous maps as \(2\)\nb-arrows.
  The injectivity assumption is only needed for the homomorphism to
  \(\Cst\)\nb-algebras in \longref{Section}{sec:to_Cstar}.
\end{remark}

The composition of arrows \(\Gr[H]\leftarrow \Gr\leftarrow \Gr[K]\)
in~\(\Grcat\) is the construction~\(\Grcomp_{\Gr}\).  Let
\(\Bisp_1,\Bisp_2\colon \Gr[H] \leftleftarrows \Gr\) and
\(\Bisp[Y]_1,\Bisp[Y]_2\colon \Gr \leftleftarrows \Gr[K]\) be
groupoid correspondences and let
\(\alpha\colon \Bisp_1 \Rightarrow \Bisp_2\) and
\(\beta\colon \Bisp[Y]_1 \Rightarrow \Bisp[Y]_2\) be
\(2\)\nb-arrows.  These induce a map
\[
  \alpha\Grcomp_{\Gr} \beta \colon
  \Bisp_1 \Grcomp_{\Gr} \Bisp[Y]_1 \Rightarrow \Bisp_2 \Grcomp_{\Gr}
  \Bisp[Y]_2,
  \qquad
  [(x,y)] \mapsto [(\alpha(x),\beta(y)],
\]
which inherits the properties of being injective,
\(\Gr[H],\Gr[K]\)-equivariant and continuous.  In addition, this
construction is ``functorial'', that is, the composition is a
bifunctor
\[
  \Grcomp_{\Gr}\colon
  \Grcat(\Gr[H],\Gr) \times \Grcat(\Gr,\Gr[K])
  \to \Grcat(\Gr[H],\Gr[K]).
\]
For each groupoid~\(\Gr\), the identity groupoid
correspondence~\(1_{\Gr}\) on~\(\Gr\) is the arrow space of~\(\Gr\)
--~which we also denote by~\(\Gr\)~-- with the obvious left and
right actions of~\(\Gr\) by multiplication; its anchor maps are the
range and source maps \(\rg,\s\colon \Gr\rightrightarrows \Gr^0\).
The following two easy lemmas describe the natural \(2\)\nb-arrows
that complete the bicategory structure of~\(\Grcat\):

\begin{lemma}
  \label{lem:unitor}
  Let \(\Bisp\colon \Gr[H]\leftarrow \Gr\) be a groupoid
  correspondence.  The maps
  \begin{alignat*}{2}
    \Gr[H]\Grcomp_{\Gr[H]}\Bisp&\to \Bisp,&\qquad
    [h,x]&\mapsto h\cdot x,\\
    \Bisp\Grcomp_{\Gr} \Gr&\to \Bisp,&\qquad
    [x,g]&\mapsto x\cdot g,
  \end{alignat*}
  are \(\Gr[H],\Gr\)\nb-equivariant homeomorphisms, which are
  natural for \(\Gr[H],\Gr\)-equivariant continuous maps
  \(\Bisp \to \Bisp'\).
\end{lemma}

\begin{proof}
  It is easy to see that these multiplication maps are bijective,
  continuous and \(\Gr[H],\Gr\)-equivariant.  Then they are
  isomorphisms of correspondences by
  \longref{Lemma}{lem:2-arrow_local_homeo}.  The naturality statement is
  obvious.
\end{proof}

\begin{lemma}
  \label{lem:associator}
  Let \(\Gr_i\) for \(1\le i \le 4\) be étale groupoids.  Let
  \(\Bisp_i\colon \Gr_i \leftarrow \Gr_{i+1}\) for \(1\le i \le 3\)
  be correspondences.  The map
  \[
    \mathrm{assoc}\colon \Bisp_1\Grcomp_{\Gr_1} (\Bisp_2\Grcomp_{\Gr_2}\Bisp_3)
    \to (\Bisp_1\Grcomp_{\Gr_1} \Bisp_2)\Grcomp_{\Gr_2} \Bisp_3,\qquad
    [x_1,[x_2,x_3]]\mapsto [[x_1,x_2],x_3],
  \]
  is a \(\Gr_1,\Gr_4\)-equivariant homeomorphism, which is natural
  with respect to \(\Gr_i,\Gr_{i+1}\)-equivariant continuous maps
  \(\alpha_i\colon \Bisp_i \to \Bisp'_i\) for \(1\le i \le 3\); that
  is, the following square commutes:
  \[
    \begin{tikzcd}[column sep=6em]
      \Bisp_1\Grcomp_{\Gr_1} (\Bisp_2\Grcomp_{\Gr_2}\Bisp_3)
      \ar[r, Rightarrow, "\alpha_1 \Grcomp_{\Gr_1} (\alpha_2\Grcomp_{\Gr_2}\alpha_3)"]
      \ar[d, Rightarrow,  "\mathrm{assoc}"] &
      \Bisp_1'\Grcomp_{\Gr_1} (\Bisp_2'\Grcomp_{\Gr_2}\Bisp_3')
      \ar[d, Rightarrow,  "\mathrm{assoc}"] \\
      (\Bisp_1\Grcomp_{\Gr_1} \Bisp_2)\Grcomp_{\Gr_2} \Bisp_3
      \ar[r, Rightarrow,  "(\alpha_1\Grcomp_{\Gr_1} \alpha_2)\Grcomp_{\Gr_2}
      \alpha_3"] &
      (\Bisp_1'\Grcomp_{\Gr_1} \Bisp_2')\Grcomp_{\Gr_2} \Bisp_3'.
    \end{tikzcd}
  \]
\end{lemma}

\begin{proposition}
  \label{pro:groupoid_bicategory}
  The data above defines a bicategory~\(\Grcat\).
\end{proposition}

\begin{proof}
  It is trivial to check that the coherence diagrams for a
  bicategory commute (these diagrams are shown, for instance,
  in \cites{Benabou:Bicategories, Leinster:Basic_Bicategories}).
\end{proof}

\begin{remark}
  General properties of groupoids, groupoid actions and groupoid
  principal bundles are shown in \cites{Arabidze:Thesis,
    Meyer-Zhu:Groupoids} in a rather abstract setting.  When we
  apply the definitions and results in~\cite{Arabidze:Thesis} to the
  category of locally compact, topological spaces with local
  homeomorphisms as partial covers, we also get a construction of a
  bicategory of étale groupoids and étale groupoid correspondences.
  Here we only require for an étale groupoid correspondence that the
  right action should be basic and that its anchor map should be a
  local homeomorphism.  In this article, we have added the
  assumptions that the object spaces of our groupoids and the orbit
  spaces of our groupoid correspondences for the right actions should be
  Hausdorff.  Both assumptions are needed for the homomorphism to
  \(\Cst\)\nb-algebras in \longref{Section}{sec:to_Cstar}.  The fact that these
  Hausdorffness assumptions define a subbicategory does not follow
  from the general theory in~\cite{Meyer-Zhu:Groupoids}, unless we
  also require all arrow spaces of groupoids to be Hausdorff.  This,
  however, is too much for certain applications.
\end{remark}

To prepare for the next section, we briefly recall how to define the
bicategory \(\Corr\) of \(\Cst\)\nb-correspondences
(see~\cite{Buss-Meyer-Zhu:Higher_twisted}).  Its objects are
\(\Cst\)\nb-algebras.  Its arrows \(A\leftarrow B\) for two
\(\Cst\)\nb-algebras \(A\) and~\(B\) are
\(A\)-\(B\)-correspondences, that is, Hilbert \(B\)\nb-modules with
a nondegenerate \Star{}representation of~\(A\) by adjointable
operators.  If \(\Hilm,\Hilm[F]\colon A \leftleftarrows B\) are two
such \(\Cst\)\nb-correspondences, then a \(2\)\nb-arrow
\(\alpha\colon \Hilm \Rightarrow \Hilm[F]\) is an \(A,B\)-bimodule
map \(\alpha\colon \Hilm \to \Hilm[F]\) that is isometric in the
sense that \(\braket{\alpha(x)}{\alpha(y)} = \braket{x}{y}\) for all
\(x,y\in \Hilm\).  These isometric bimodule maps are composed as
maps.  This makes the arrows \(A\leftarrow B\) and their
\(2\)\nb-arrows into a category.  The composition of arrows is the
usual completed tensor product of \(\Cst\)\nb-correspondences.  This
is clearly a bifunctor for isometric bimodule maps, as
required for a bicategory.

The unit arrow~\(1_A\) for a \(\Cst\)\nb-algebra is~\(A\) itself,
viewed as a correspondence \(A\leftarrow A\) in the obvious way,
using the bimodule structure by left and right multiplication and
the \(A\)\nb-valued inner product \(\braket{x}{y} \defeq x^* y\) for
\(x,y\in A\).  There are natural isomorphisms of
\(\Cst\)\nb-correspondences
\[
  A\otimes_A \Hilm \cong \Hilm,\qquad
  \Hilm \otimes_B B \cong \Hilm,\qquad
  (\Hilm \otimes_B \Hilm[F]) \otimes_C \Hilm[G]
  \cong \Hilm \otimes_B (\Hilm[F] \otimes_C \Hilm[G])
\]
for three composable \(\Cst\)\nb-correspondences
\(\Hilm\colon A\leftarrow B\), \(\Hilm[F]\colon B\leftarrow C\),
\(\Hilm[G]\colon C\leftarrow D\).  It is easy to check that these
are natural with respect to the \(2\)\nb-arrows above and make the
diagrams required for a bicategory commute (see
\cites{Benabou:Bicategories, Leinster:Basic_Bicategories}).

Note that \(\Corr\) is defined
in~\cite{Buss-Meyer-Zhu:Higher_twisted} using only isomorphisms
of \(\Cst\)\nb-correspondences as \(2\)\nb-arrows.  Here we allow
isometries that are not invertible, not even adjointable.  This is
useful in some situations.  In particular, non-invertible isometries
of \(\Cst\)\nb-\alb{}correspondences are crucial
in~\cites{Meyer:Unbounded} or to treat partial actions in
bicategorical terms.  In this article, we allow \(2\)\nb-arrows that
are not invertible because this does not cause any extra problems
and it allows us to prove a stronger statement.

A \(\Cst\)\nb-correspondence \(\Hilm\colon A\leftarrow B\) is called
\emph{proper} if the left action factors through the ideal of
compact operators, \(A\to \Comp(\Hilm)\).  The proper
\(\Cst\)\nb-correspondences form a subbicategory \(\Corr_\prop\),
that is, identity correspondences are proper and the composite of
two proper correspondences is again proper.

\section{The homomorphism to \texorpdfstring{$\Cst$}{C*}-algebras}
\label{sec:to_Cstar}

In this section, we first recall how to define the
\(\Cst\)\nb-algebra of an (étale) groupoid.  Then we turn a groupoid
correspondence into a \(\Cst\)\nb-correspondence between the
groupoid \(\Cst\)\nb-algebras.  We prove that this is part of a
homomorphism of bicategories \(\Grcat\to\Corr\) between the
bicategories of groupoid and \(\Cst\)\nb-correspondences.

Both \(\Cst(\Gr)\) for a groupoid~\(\Gr\) and \(\Cst(\Bisp)\) for a
groupoid correspondence~\(\Bisp\) are defined using a certain space
of ``quasi-continuous'' functions.  We define this in the generality
of a locally compact, locally Hausdorff space~\(X\) to treat both
cases simultaneously.  If $V\subseteq X$ is open and Hausdorff, then
we extend a function \(f\in\Contc(V)\) to~\(X\) by letting
\(f(g)=0\) for all \(g\in X\setminus V\).  Let $\ContS(X)$ be the
linear span of functions on~\(X\) of this form.  If~\(X\) is
Hausdorff, then $\ContS(X)$ is equal to the space \(\Contc(X)\) of
continuous, compactly supported functions on~\(X\).  If~\(X\) is not
Hausdorff, then there may be too few continuous functions on it, and
we have to use the space \(\ContS(X)\) instead.  The following
proposition allows us to build \(\ContS(X)\) using a specific
covering by locally compact Hausdorff open subsets:

\begin{proposition}[\cite{Exel:Inverse_combinatorial}*{Proposition~3.10}]
  \label{pro:covering_decomposition}
  Let~\(X\) be a topological space.  Let~\(U_i\) for \(i\in I\) be
  open subsets that are locally compact and Hausdorff.  Assume
  \(X= \bigcup_{i\in I} U_i\).  Then $\ContS(X)$ is equal to the
  linear span of the subspaces $\Contc(U_i)$ for \(i\in I\).
\end{proposition}

\begin{proof}
  Let $V$ be a locally compact Hausdorff open subset of~$X$ and
  $f\in \Contc(V)$.  Since the support of~$f$ is compact, it is
  covered by finitely many of the open subsets~\(U_i\).  In
  addition, this covering has a finite partition of unity.  Then we
  may write \(f = \sum_{j=1}^n f_j\) with
  \(f_j \in \Contc(U_{i_j})\).
\end{proof}

\begin{definition}
  A \emph{slice} of a groupoid~\(\Gr\) is an open subset
  \(V\subseteq\Gr\) such that \(\s|_V\) and~\(\rg|_V\) are
  injective.  Let \(\Gr\) and~\(\Gr[H]\) be groupoids.  A
  \emph{slice} of a groupoid correspondence
  \(\Bisp\colon \Gr[H]\leftarrow\Gr\) is an open subset
  \(V\subseteq\Bisp\) such that \(\s|_V\colon V\to \Gr^0\) and
  \(\Qu|_V\colon V\to \Bisp/\Gr\) are injective.
\end{definition}

The name ``slice'' comes from~\cite{Exel:Inverse_combinatorial}.  We
find this name more friendly than the more common name
``bisections''.

When we view a
groupoid~\(\Gr\) as the identity correspondence over itself, then
the range map induces a homeomorphism \(\Gr/\Gr \cong \Gr^0\).
Therefore, \(\Qu\colon \Gr \to \Gr/\Gr\) is equivalent to the range
map and~\(\Gr\) as a groupoid and as a groupoid correspondence have
the same slices.  If~\(\Bisp\) is a groupoid correspondence,
then \(\s\colon \Bisp \to \Gr^0\) is a local homeomorphism by
assumption and \(\Qu\colon \Bisp\to\Bisp/\Gr\) is one by
\longref{Lemma}{lem:basic_orbit_lh}.  Therefore, any point in~\(\Bisp\)
has an open neighbourhood that is a slice.  In other words, the
slices cover~\(\Bisp\).  Then
\longref{Proposition}{pro:covering_decomposition} allows us to write any
element of \(\ContS(\Bisp)\) as a finite sum \(\sum_{i=1}^n f_i\)
for functions \(f_i \in \Contc(V_i)\) and slices~\(V_i\) for
\(i=1,\dotsc,n\).  The special case when \(\Bisp=1_{\Gr} \cong \Gr\)
is itself a groupoid is already well known.

We define a \Star{}algebra structure on $\ContS(\Gr)$ as
in~\cite{Khoshkam-Skandalis:Regular}.  Namely, if
$\xi, \eta \in \ContS(\Gr)$, let
\begin{align*}
  \xi * \eta(g) &= \sum_{h \in \Gr^{\rg(g)}} \xi(h)\eta(h^{-1}g),\\
  \xi^*(g) &= \overline{\xi(g^{-1})}.
\end{align*}
We recall why this is well defined.  Assume $\xi\in\Contc(V)$ and
$\eta\in\Contc(W)$ for slices $V$ and~$W$.  Then
\(V\cdot W \defeq \setgiven{g\cdot h}{g\in V,\ h\in W,\
  \s(g)=\rg(h)}\) is a slice as well, and so is
\(V^* \defeq \setgiven{g^{-1}}{g\in V}\).  The formula for
\(\xi*\eta\) above simplifies to \(\xi * \eta(g) = \xi(h)\eta(k)\)
if there are $h \in V$ and $k \in W$ with \(h\cdot k = g\), and
\(\xi*\eta(g)=0\) otherwise.  As a result,
\(\xi*\eta \in \Contc(V\cdot W) \subseteq \ContS(\Gr)\).  Since
\(\ContS(\Gr)\) is spanned by functions in \(\Contc(V)\) for
slices~\(V\), this implies that \(\xi*\eta \in \ContS(\Gr)\) for
all \(\xi,\eta\in \ContS(\Gr)\).  Similarly,
$\xi^* \in \Contc(V^{-1})$ if \(\xi\in \Contc(V)\), and then
\(\xi^*\in \ContS(\Gr)\) for all \(\xi\in \ContS(\Gr)\).

Routine computations show that the convolution above is bilinear and
associative and that \(\xi\mapsto \xi^*\) is a conjugate-linear
involution with \((\xi*\eta)^* = \eta^* * \xi^*\).  Thus
\(\ContS(\Gr)\) is a \Star{}algebra.

Next, we recall why a maximal \(\Cst\)\nb-seminorm
on~\(\ContS(\Gr)\) exists.  The subset \(\Gr^0\subseteq \Gr\) is a
slice, called the unit slice.  So
\(\Contc(\Gr^0) \subseteq \ContS(\Gr)\).  The convolution and
involution on \(\ContS(\Gr)\) restrict to the usual pointwise
multiplication and pointwise involution on \(\Contc(\Gr^0)\).  Since
\(\Contc(\Gr^0)\) is the union of the \(\Cst\)\nb-subalgebras
\(\Cont_0(U)\) for relatively compact, open subsets
\(U\subseteq \Gr^0\), any \Star{}representation of \(\Contc(\Gr^0)\)
on a Hilbert space is bounded by the supremum norm.  Therefore, any
\(\Cst\)\nb-seminorm on \(\Contc(\Gr^0)\) is bounded by the usual
supremum norm.
If \(\xi\in \Contc(V)\) for a slice~\(V\), then
\(\xi * \xi^* \in \Contc(V\cdot V^{-1}) \subseteq \Contc(\Gr^0)\),
and so
\[
  \norm{\xi}
  = \norm{\xi*\xi^*}^{1/2}
  \le \norm{\xi*\xi^*}_\infty^{1/2}
  = \norm{\xi}_\infty.
\]
Any element \(\xi\in\ContS(\Gr)\) is a finite linear combination of
such functions on slices.  Therefore, there is \(C>0\) with
\(\norm{\xi}\le C\) for all \(\Cst\)\nb-seminorms
on~\(\ContS(\Gr)\).  Therefore, the supremum of the set of
\(\Cst\)\nb-seminorms on~\(\ContS(\Gr)\) exists.  This is again a
\(\Cst\)\nb-seminorm, and clearly the maximal such.  Actually, it is
a \(\Cst\)\nb-norm and not just a \(\Cst\)\nb-seminorm because of
the regular representations, but this will not be crucial in what
follows.

\begin{definition}
  The groupoid $\Cst$-algebra $\Cst(\Gr)$ of~$\Gr$ is the completion
  of \(\ContS(\Gr)\) in the largest \(\Cst\)\nb-norm.
\end{definition}

Now let \(\Gr\) and~\(\Gr[H]\) be groupoids and let
$\Bisp\colon\Gr[H]\leftarrow\Gr$ be a groupoid correspondence.  We
are going to build a \(\Cst\)\nb-correspondence
$\Cst(\Bisp)\colon\Cst(\Gr[H])\leftarrow\Cst(\Gr)$.  That is,
\(\Cst(\Bisp)\) is a Hilbert \(\Cst(\Gr)\)-module with a
nondegenerate left action of \(\Cst(\Gr[H])\).  The first result of
this kind for Morita equivalences of Hausdorff locally compact
groupoids with Haar systems was proven by Muhly--Renault--Williams
in~\cite{Muhly-Renault-Williams:Equivalence}.  The Hausdorffness
assumption was soon removed by Renault (see
\cite{Renault:Representations}*{Corollaire~5.4}).  More general
groupoid correspondences between Hausdorff locally compact groupoids
with Haar systems were treated by Holkar~\cite{Holkar:Thesis}.  But
we are not aware that this construction has been carried over to
groupoid correspondences between possibly non-Hausdorff étale
groupoids as we need it.  Therefore, we give the details of this
construction.  Since we only work in the étale case, this is much
easier than the constructions in the papers mentioned above.

We will define \(\Cst(\Bisp)\) as a completion of \(\ContS(\Bisp)\)
in a suitable norm.  The algebraic structure of the correspondence
is easy to write down on the dense subspaces of
\(\ContS\)-functions.  Namely, let $\xi$, $\eta \in \ContS(\Bisp)$,
$\gamma \in \ContS(\Gr)$, $\zeta \in \ContS(\Gr[H])$ and
\(x\in\Bisp\), \(g\in\Gr\).  Then we define
\begin{align}
  \label{eq:convolution_XG}
  \xi * \gamma (x)
  &= \sum_{g\in \Gr^{\s(x)}} \xi(x\cdot g) \gamma(g^{-1}),\\
  \label{eq:braket_XX}
  \braket{\xi}{\eta}(g)
  &= \sum_{\setgiven{x \in \Bisp}{\s(x)=\rg(g)}}
    \overline{\xi(x)}\eta(x\cdot g),\\
  \label{eq:convolution_GX}
  \zeta * \xi(x)
  &= \sum_{h \in \Gr[H]^{\rg(x)}} \zeta(h)\xi(h^{-1}x).
\end{align}
We first check that this is well defined, that is, that these
functions are finite linear combinations of \(\Contc\)\nb-functions
on slices.

\begin{lemma}
  \label{lem:corr_structure}
  Let \(\xi\in\Contc(V_1)\), \(\eta\in\Contc(V_2)\),
  \(\gamma\in\Contc(W)\), and \(\zeta\in\Contc(Z)\) for slices
  \(V_1,V_2\subseteq \Bisp\), \(W\subseteq \Gr\) and
  \(Z\subseteq \Gr[H]\).  The following subsets are also slices:
  \begin{align*}
    V_1 W
    &\defeq \setgiven{x g}{x\in V_1,\ g\in W,\ \s(x) = \rg(g)}
      \subseteq \Bisp,\\
    \braket{V_1}{V_2}
    &\defeq \setgiven{\braket{x_1}{x_2}}{x_1\in V_1,\ x_2\in V_2,\ \Qu(x_1)=\Qu(x_2)}
      \subseteq \Gr,\\
    Z V_1
    &\defeq \setgiven{h x}{h\in Z,\ x\in V_1,\ \s(h) = \rg(x)}
      \subseteq \Bisp.
  \end{align*}
  For \(y\in V_1 W\), there are unique \(x\in V_1\), \(g\in W\) with
  \(x g = y\), and \(\xi*\gamma \in \Contc(V_1 W)\) with
  \((\xi*\gamma)(y) = \xi(x) \gamma(g)\).  For
  \(g\in \braket{V_1}{V_2}\), there are unique \(x_1\in V_1\),
  \(x_2\in V_2\) with \(g= \braket{x_1}{x_2}\), and
  \(\braket{\xi}{\eta} \in \Contc(\braket{V_1}{V_2})\) with
  \(\braket{\xi}{\eta}(g) = \overline{\xi(x_1)} \eta(x_2)\).  For
  \(y\in Z V_1\), there are unique \(h\in Z\), \(x\in V_1\) with
  \(y= h x\), and \(\zeta*\xi \in \Contc(Z V_1)\) with
  \((\zeta*\xi)(y) = \zeta(h) \xi(x)\).
\end{lemma}

\begin{proof}
  All three assertions are proven similarly, and the claims about
  \(V_1 W\) and \(Z V_1\) are, in fact, special cases of
  \longref{Lemma}{lem:multiply_slices} and
  \longref{Proposition}{hcomp} about the composition of
  correspondences later on.  Therefore, we only write down the proof
  for \(\braket{V_1}{V_2}\).

  The subset \(\braket{V_1}{V_2}\) is open by
  \longref{Proposition}{pro:innprod}.  Let \(x_1,y_1\in V_1\) and
  \(x_2,y_2\in V_2\) satisfy \(\Qu(x_1) = \Qu(x_2)\) and
  \(\Qu(y_1) = \Qu(y_2)\), so that \(\braket{x_1}{x_2}\) and
  \(\braket{y_1}{y_2}\) are defined.  Assume first that
  \(\s(\braket{x_1}{x_2}) = \s(\braket{y_1}{y_2})\).  Then
  \longref{Lemma}{def:innerprodofcorr} implies
  \(\s(x_2) = \s(\braket{x_1}{x_2}) = \s(\braket{y_1}{y_2}) =
  \s(y_2)\).  Then \(x_2 = y_2\) because~\(V_2\) is a slice
  of~\(\Bisp\).  Then \(\Qu(x_1) = \Qu(x_2) = \Qu(y_2) = \Qu(y_1)\).
  Then \(x_1 = y_1\) because~\(V_1\) is a slice of~\(\Bisp\).
  This proves both that~\(\s\) is injective on \(\braket{V_1}{V_2}\)
  and that the elements \(x_1\in V_1\), \(x_2\in V_2\) with
  \(g= \braket{x_1}{x_2}\) for \(g\in \braket{V_1}{V_2}\) are
  unique.  A similar argument shows that \(x_1 = y_1\) and
  \(x_2 = y_2\) hold if
  \(\rg(\braket{x_1}{x_2}) = \rg(\braket{y_1}{y_2})\).  Therefore,
  \(\braket{V_1}{V_2}\) is a slice.  It follows easily
  from~\eqref{eq:braket_XX} that this function is only nonzero in
  \(\braket{V_1}{V_2}\), that is, for \(g\in G\) for which there are
  \(x_1\in V_1\), \(x_2\in V_2\) --~necessarily unique~-- with
  \(x_1 g = x_2\) and that
  \(\braket{\xi}{\eta}(g) = \conj{\xi(x_1)} \eta(x_2)\).  Thus
  \(\braket{\xi}{\eta} \in \Contc(\braket{V_1}{V_2})\).
\end{proof}

\begin{lemma}
  \label{lem:ContS_algebraic}
  The multiplication maps above make \(\ContS(\Bisp)\) a
  \(\ContS(\Gr[H])\)-\(\ContS(\Gr)\)-bimodule, that is, they are
  bilinear and
  \(\xi* (\gamma_1*\gamma_2) = (\xi* \gamma_1) * \gamma_2\),
  \((\zeta_1 * \zeta_2) *\xi = \zeta_1 * (\zeta_2 *\xi)\),
  \((\zeta*\xi)*\gamma = \zeta *(\xi*\gamma)\) for
  \(\xi\in\ContS(\Bisp)\),
  \(\gamma,\gamma_1,\gamma_2\in \ContS(\Gr)\),
  \(\zeta,\zeta_1,\zeta_2\in \ContS(\Gr[H])\).  The inner product is
  linear in the second variable and satisfies
  \((\braket{\xi}{\eta})^* = \braket{\eta}{\xi}\),
  \(\braket{\xi}{\eta*\gamma} = \braket{\xi}{\eta} * \gamma\), and
  \(\braket{\zeta*\xi}{\eta} = \braket{\xi}{\zeta^* * \eta}\) for
  \(\xi,\eta\in\ContS(\Bisp)\), \(\gamma\in \ContS(\Gr)\),
  \(\zeta\in \ContS(\Gr[H])\).  It follows that the inner product is
  conjugate-linear in the first variable and satisfies
  \(\braket{\xi*\gamma}{\eta} = \gamma^* * \braket{\xi}{\eta}\) for
  \(\xi,\eta\in\ContS(\Bisp)\), \(\gamma\in \ContS(\Gr)\).
\end{lemma}

\begin{proof}
  All claims for functions in~\(\ContS\) follow if they hold for
  compactly supported continuous functions on slices.  In this
  special case, they follow from the associativity of our products
  or from the properties of the bracket operation in
  \longref{Lemma}{def:innerprodofcorr}.
\end{proof}

\begin{lemma}
  \label{lem:positive_innerprod}
  Let \(\xi\in\ContS(\Bisp)\).  Then there are finitely many
  elements \(a_i\in \ContS(\Gr)\) for \(i=1,\dotsc,n\) with
  \(\braket{\xi}{\xi} = \sum_{i=1}^n a_i * a_i^*\); roughly speaking
  \(\braket{\xi}{\xi} \ge 0\) in \(\ContS(\Bisp)\).  In addition, if
  \(\xi\neq0\), then \(\braket{\xi}{\xi}\neq0\).
\end{lemma}

\begin{proof}
  Write \(\xi = \sum_{k=1}^n \xi_k\) with \(\xi_k\in \Contc(V_k)\)
  for slices \(V_k \subseteq \Bisp\) for \(k=1,\dotsc,n\).  Let
  \(\Qu\colon \Bisp \to \Bisp/\Gr\) be the quotient map.  Recall
  that the space~\(\Bisp/\Gr\) is locally compact and Hausdorff by
  \longref{Proposition}{pro:basic_actions_free_proper}.  The subset
  \(K \defeq \bigcup_{k=1}^n \Qu(\supp \xi_k) \subseteq \Bisp/\Gr\)
  is compact, hence paracompact.  Then there are functions
  \(\varphi_i'\in \Contc(\Qu(V_i))\) for \(i=1,\dotsc,n\) with
  \(\sum_{i=1}^n \abs{\varphi_i'(y)}^2 = 1\) for all \(y\in K\).
  Since each~\(\Qu\) is a slice,
  \(\Qu|_{V_i} \colon V_i \to \Qu(V_i)\) is a homeomorphism.  Let
  \(\varphi_i \defeq \varphi'_i \circ \Qu|_{V_i}^{-1}\in
  \Contc(V_i)\).  We claim that
  \(a_i \defeq \braket{\xi}{\varphi_i}\) for \(i=1,\dotsc,n\) will
  do the job.

  To prove the claim, we first compute
  \(\varphi_i * \braket{\varphi_i}{\xi_k} \in \ContS(\Bisp)\).
  Since all functions involved are continuous with compact support
  on certain slices of~\(\Bisp\), so is this combination by the
  proof of \longref{Lemma}{lem:corr_structure}.  Its support is
  \(V_i \cdot \braket{V_i}{V_k} \subseteq \Bisp\).  Any element of
  this slice is of the form \(v_1 \cdot \braket{v_2}{v_3}\) for
  some \(v_1,v_2\in V_i\), \(v_3\in V_k\), for which this expression
  is defined; that is, \(\Qu(v_2) = \Qu(v_3)\) and
  \(\s(v_1) = \rg(\braket{v_2}{v_3}) = \s(v_2)\).  And then
  \[
    \varphi_i * \braket{\varphi_i}{\xi_k}
    (v_1 \cdot \braket{v_2}{v_3})
    = \varphi_i(v_1) \conj{\varphi_i(v_2)} \xi_k(v_3).
  \]
  Since \(v_1,v_2\in V_i\) and~\(V_i\) is a slice of~\(\Bisp\),
  \(\s(v_1) = \s(v_2)\) implies \(v_1 = v_2\).  Then
  \(v_1 \cdot \braket{v_2}{v_3} = v_3\) by
  \longref{Lemma}{def:innerprodofcorr}.  As a result,
  \(\varphi_i * \braket{\varphi_i}{\xi_k} \in \Contc(V_k)\) and
  if \(x\in V_k\), then
  \[
    \varphi_i * \braket{\varphi_i}{\xi_k} (x)
    = \abs{\varphi_i(a)}^2 \xi_k(x)
  \]
  if there is \(a\in V_i\) with \(\Qu(a) = \Qu(x)\), and~\(0\)
  otherwise.  Then
  \(\abs{\varphi_i(a)}^2 = \abs{\varphi'_i(\Qu(x))}^2\).  Since
  \(\sum \abs{\varphi'_i(\Qu(x))}^2 = 1\) for all \(x\in V_k\) with
  \(\xi_k(x)\neq0\), this implies
  \[
    \sum_{i=1}^n \varphi_i * \braket{\varphi_i}{\xi_k} = \xi_k.
  \]
  Since this holds for all~\(k\), summing over~\(k\) gives
  \(\sum_{i=1}^n \varphi_i * \braket{\varphi_i}{\xi} = \xi\).
  Therefore,
  \[
    \braket{\xi}{\xi}
    = \braket*{\xi}{\sum_{i=1}^n \varphi_i *
      \braket{\varphi_i}{\xi}}
    = \sum_{i=1}^n \braket{\xi}{\varphi_i}* \braket{\varphi_i}{\xi}
    = \sum_{i=1}^n a_i * a_i^*
  \]
  as desired.

  If \(y\in \Gr^0\), then we compute
  \[
    \braket{\xi}{\xi}(1_y)
    = \sum_{\setgiven{x\in\Bisp}{\s(x) = y}} \conj{\xi(x)}{\xi(x)}
    = \sum_{\setgiven{x\in\Bisp}{\s(x) = y}} \abs{\xi(x)}^2.
  \]
  If this vanishes for all \(y\in\Gr^0\), then \(\xi(x)=0\) for all
  \(x\in\Bisp\), and then \(\xi=0\).
\end{proof}

\longref{Lemma}{lem:positive_innerprod} implies the following:
\begin{itemize}
\item
  \(\norm{\xi} \defeq \norm{\braket{\xi}{\xi}}_{\Cst(\Gr)}^{1/2}\)
  is a norm on \(\ContS(\Bisp)\);
\item the given inner product and right \(\ContS(\Gr)\)-module
  structure on \(\ContS(\Bisp)\) extend to a Hilbert
  \(\Cst(\Gr)\)-module structure on the norm completion of
  \(\ContS(\Bisp)\).
\end{itemize}
We denote this Hilbert \(\Cst(\Gr)\)-module by \(\Cst(\Bisp)\).

\begin{lemma}
  \label{lem:left_action_bounded}
  If \(\zeta\in \ContS(\Gr[H])\), then the map
  \(\xi\mapsto \zeta*\xi\) is bounded for the norm on
  \(\ContS(\Bisp)\) defined above.
\end{lemma}

\begin{proof}
  First, let \(\zeta\in \Contc(\Gr[H]^0) \subseteq \ContS(\Gr[H])\).
  Then \((\zeta*\xi)(x) = \zeta(\rg(x))\cdot \xi(x)\) for all
  \(\xi\in \ContS(\Bisp)\).  Let~\(M\) be the maximum
  of~\(\abs{\zeta}\) and let
  \(\tau(y) \defeq \sqrt{M^2-\abs{\zeta(y)}^2}\) for all
  \(y\in\Gr[H]^0\).  This defines a bounded function on~\(\Gr[H]^0\)
  with \(\zeta^* * \zeta + \tau^* * \tau = M^2\).  Therefore, we may
  estimate
  \[
    M^2\cdot \braket{\xi}{\xi}
    = \braket{\xi}{(\zeta^* * \zeta + \tau^* * \tau) * \xi}.
    = \braket{\zeta *\xi}{\zeta * \xi}
    + \braket{\tau * \xi}{\tau * \xi}
    \ge \braket{\zeta * \xi}{\zeta * \xi}
  \]
  This inequality in \(\ContS(\Gr) \subseteq \Cst(\Gr)\) implies
  \(M \cdot \norm{\xi} \ge \norm{\zeta*\xi}\).  Equivalently,
  the operator norm of \(\zeta\in \Cont(\Gr[H]^0)\) is at most
  \(\norm{\zeta}_\infty\).

  Next, if \(\zeta\in \Contc(Z)\) for a slice
  \(Z\subseteq \Gr[H]\), then
  \(\norm{\zeta*\xi}^2 = \norm{\braket{\xi}{\zeta^**\zeta *\xi}}\),
  and \(\zeta^**\zeta\in \Contc(\Gr^0)\).  Since the latter has
  operator norm at most \(\norm{\zeta}^2_\infty<\infty\), it follows
  that the operator norm of left convolution with~\(\zeta\) is at
  most~\(\norm{\zeta}_\infty\).  Finally,
  \longref{Proposition}{pro:covering_decomposition} reduces the case of
  general \(\zeta\in \ContS(\Gr[H])\) to this special case.
\end{proof}

\longref{Lemma}{lem:left_action_bounded} implies that the representation
of \(\ContS(\Gr[H])\) on \(\ContS(\Bisp)\) by left convolution
extends to a representation by bounded linear operators on
\(\Cst(\Bisp)\).  This representation inherits the algebraic
properties in \longref{Lemma}{lem:ContS_algebraic} by continuity.
Therefore, we get a \Star{}homomorphism from \(\ContS(\Gr[H])\) to
the \(\Cst\)\nb-algebra of adjointable operators on~\(\Cst(\Bisp)\).
This extends uniquely to a \Star{}homomorphism on \(\Cst(\Gr[H])\)
by the universal property of the \(\Cst\)\nb-completion.  The
computations above also show that the left
\(\Contc(\Gr[H]^0)\)-module structure on \(\ContS(\Bisp)\) is
nondegenerate.  This implies that the representation of
\(\Cst(\Gr[H])\) on \(\Cst(\Bisp)\) is nondegenerate.  This
completes the construction of the
\(\Cst(\Gr[H])\)-\(\Cst(\Gr)\)-correspondence \(\Cst(\Bisp)\) from a
groupoid correspondence \(\Bisp\colon \Gr[H] \leftarrow \Gr\).

Before we continue to build the homomorphism \(\Grcat\to\Corr\), we
examine \(\Cst(\Bisp)\) for the groupoid correspondences between
spaces, groups and transformation groups.  As expected, we get the
\(\Cst\)\nb-correspondences whose Cuntz--Pimsner algebras (in the
sense of Katsura) are the topological graph algebras, Nekrashevych's
\(\Cst\)\nb-algebra of a self-similar group, and the Exel--Pardo
\(\Cst\)\nb-algebra of a self-similar graph.

\begin{example}
  Let \(\Gr=\Gr[H]\) be a locally compact space.  In
  \longref{Example}{exa:topological_graph}, we have identified a groupoid
  correspondence \(\Bisp\colon \Gr\leftarrow\Gr\) with a topological
  graph.  When we pass to \(\Cst\)\nb-algebras, we get
  \(\Cst(\Gr) = \Cst(\Gr[H]) = \Cont_0(\Gr)\).  Since~\(\Bisp\) is
  Hausdorff, \(\ContS(\Bisp) = \Contc(\Bisp)\).  The
  \(\Contc(\Gr^0)\)-bimodule structure on \(\Contc(\Bisp)\) extends
  continuously to the \(\Cont_0(\Gr^0)\)-bimodule structure
  \((f\cdot \xi\cdot g)(x) = f(\rg(x))\cdot \xi(x) \cdot g(\s(x))\)
  for \(f,g\in\Cont_0(\Gr)\), \(\xi\in\Contc(\Bisp)\),
  \(x\in \Bisp\).  The inner product simplifies to
  \(\braket{\xi}{\eta}(y) = \sum_{\s(x) = y} \conj{\xi(x)} \eta(x)\)
  for \(\xi,\eta\in\Contc(\Bisp)\), \(y\in\Gr\).  The norm
  completion \(\Cst(\Bisp)\) of this is exactly the
  \(\Cst\)\nb-correspondence that is used by
  Katsura~\cite{Katsura:class_I} to define topological graph
  \(\Cst\)\nb-algebras.
\end{example}

\begin{example}
  We have identified a proper groupoid correspondence
  \(\Bisp\colon \Gr \leftarrow \Gr\) for a (discrete) group~\(\Gr\)
  in \longref{Example}{exa:topological_graph} with the covering
  permutational bimodule of a self-similar group, minus a certain
  faithfulness property.  The assumption that~\(\Bisp\) is a proper
  correspondence says that the set \(\Bisp/\Gr\) is finite.  Here
  \(\ContS(\Gr)\) is simply the group ring \(\C[G]\), and
  \(\ContS(\Bisp)\) is the vector space \(\C[\Bisp]\) with
  basis~\(\Bisp\).  The bimodule structure and the inner product on
  \(\C[\Bisp]\) above are the same as defined by Nekrashevych in
  \cite{Nekrashevych:Cstar_selfsimilar}*{Section~3.1}.  The
  Cuntz--Pimsner algebra of the \(\Cst\)\nb-correspondence
  \(\Cst(\Bisp)\colon \Cst(\Gr) \leftarrow \Cst(\Gr)\) is
  Nekrashevych's universal Cuntz--Pimsner
  algebra~\(\mathcal{O}_{(\Gr,\Bisp)}\) of the self-similar group as
  defined in~\cite{Nekrashevych:Cstar_selfsimilar}.
\end{example}

\begin{example}
  \label{exa:Exel-Pardo}
  Now let \(\Gr=\Gr[H]= V\rtimes \Gamma\) for a group~\(\Gamma\) and
  a discrete set~\(V\).  We have described groupoid correspondences
  \(\Bisp\colon \Gr\leftarrow \Gr\) in
  \longref{Proposition}{pro:self-similar_graph}, relating them to the
  self-similar graphs of Exel and
  Pardo~\cite{Exel-Pardo:Self-similar}.  Exel and Pardo associate a
  \(\Cst\)\nb-algebra to such a self-similar graph and identify this
  \(\Cst\)\nb-algebra with the Cuntz--Pimsner algebra of a
  \(\Cst\)\nb-correspondence on \(\Cst(V\ltimes \Gamma)\).  We claim
  that this \(\Cst\)\nb-correspondence is exactly our
  \(\Cst(\Bisp)\).  To begin with, \(\ContS(\Gr)\) is the algebraic
  crossed product algebra \(\C[V] \rtimes \Gamma\), which is spanned
  by the characteristic functions~\(\delta_{v,g}\) for \(v\in V\)
  and \(g\in\Gamma\).  Similarly, \(\ContS(\Gr)\) is the vector
  space \(\C[E\times \Gamma]\) with basis \(E\times\Gamma\).
  The bimodule structure and inner product above may easily be
  expressed in terms of these bases, and the map
  \(\delta_{e,g}\mapsto t_ev_g\) gives an isomorphism from
  \(\Cst(\Bisp)\) to the \(\Cst\)\nb-correspondence that is
  denoted~\(M\) in~\cite{Exel-Pardo:Self-similar}*{Section~10}.
\end{example}

After these examples, we resume the construction of a homomorphism
\(\Grcat\to\Corr\) and turn to the action of \(2\)\nb-arrows.  Let
\(\Gr[H]\) and~\(\Gr\) be groupoids and let
\(\Bisp,\Bisp[Y]\colon \Gr[H] \leftarrow \Gr\) be groupoid
correspondences.  A \(2\)\nb-arrow
\(\alpha\colon \Bisp\Rightarrow\Bisp[Y]\) is, by definition, an
injective, \(\Gr[H]\)-\(\Gr\)-equivariant, continuous map
\(\alpha\colon \Bisp \to \Bisp[Y]\).  By
\longref{Lemma}{lem:2-arrow_local_homeo}, it follows that~\(\alpha\)
is a homeomorphism from~\(\Bisp\) onto an open subset
of~\(\Bisp[Y]\).  Then any slice of~\(\Bisp\) is also a
slice of~\(\Bisp[Y]\), and so extension by zero defines an
injective map \(\ContS(\Bisp) \subseteq \ContS(\Bisp[Y])\).  This
map preserves both the bimodule structure and the inner product.
Therefore, it is an isometry for the Hilbert module norms.  Then it
extends uniquely to an isometric map on the completions.  This
extension remains a \(\Cst(\Gr[H])\)-\(\Cst(\Gr)\)-bimodule map,
which we denote by
\(\Cst(\alpha)\colon \Cst(\Bisp) \Rightarrow \Cst(\Bisp[Y])\).

It is easy to see that \(\alpha\mapsto \Cst(\alpha)\) is functorial,
that is, the identity map \(\Bisp \congto \Bisp\) goes to the
identity map \(\Cst(\Bisp) \congto \Cst(\Bisp)\), and
\(\Cst(\alpha)\circ \Cst(\beta) = \Cst(\alpha\circ\beta)\) for
composable \(2\)\nb-arrows.

The unit correspondence~\(1_{\Gr}\) on a groupoid~\(\Gr\) is~\(\Gr\)
with the actions of~\(\Gr\) by left and right multiplication.
Inspection shows that the resulting \(\ContS(\Gr)\)-bimodule
structure and inner product on \(\ContS(\Gr)\) are given by left and
right convolution and the usual inner product formula
\(\braket{\xi}{\eta} \defeq \xi^* * \eta\).  Therefore, the
completion \(\Cst(1_{\Gr})\) is equal to \(\Cst(\Gr)\) with the
usual Hilbert bimodule structure over itself.  Thus the construction
\(\Bisp\mapsto \Cst(\Bisp)\) maps the unit groupoid correspondence
to the unit \(\Cst\)\nb-correspondence.  We briefly say that our
construction is \emph{strictly unital}.

Next, we turn to the compatibility with composition of groupoid
correspondences.  We first prove a preparatory lemma about the
composition of slices, which is analogous to
\longref{Lemma}{lem:corr_structure}.

\begin{lemma}
  \label{lem:multiply_slices}
  Let $\Bisp\colon \Gr[K] \leftarrow \Gr[H]$ and
  $\Bisp[Y]\colon \Gr[H]\leftarrow \Gr$ be composable groupoid
  correspondences.  Recall that $\Bisp\Grcomp_{\Gr[H]}\Bisp[Y]$ is
  is the orbit space of the diagonal \(\Gr[H]\)\nb-action on
  \(\Bisp\times_{\s,\Gr[H]^0,\rg}\Bisp[Y]\), with a canonical structure
  of groupoid correspondence $\Gr[K]\leftarrow\Gr$.  Denote elements
  of $\Bisp\Grcomp_{\Gr[H]}\Bisp[Y]$ as \([x,y]\) for \(x\in\Bisp\),
  \(y\in\Bisp[Y]\) with \(\s(x) = \rg(y)\).
  Let \(U\subseteq \Bisp\) and
  \(V\subseteq \Bisp[Y]\) be slices.  Define
  \[
    U\cdot V \defeq \setgiven{[x,y]}{x\in U,\ y\in V,\ \s(x) = \rg(y)}
    \subseteq \Bisp \Grcomp_{\Gr} \Bisp[Y].
  \]
  This subset is a slice.  For each point in \(z\in U\cdot V\),
  the elements \(x\in U\) and \(y\in V\) with \(z=[x,y]\) are
  unique.
\end{lemma}

\begin{proof}
  The subset \(U\times_{\Gr^0} V\) is open in
  \(\Bisp \times_{\s,\Gr^0,\rg} \Bisp[H]\) by definition.  Then its
  image \(U\cdot V\) in~\(\Bisp \Grcomp_{\Gr} \Bisp[Y]\) is open by
  \longref{Lemma}{lem:basic_orbit_lh}.  Let \(x_1,x_2\in U\) and
  \(y_1,y_2\in V\) be such that \(\s(x_1) = \rg(y_1)\) and
  \(\s(x_2) = \rg(y_2)\).  Assume first that
  \(\s[x_1,y_1] = \s[x_2,y_2]\).  This means that
  \(\s(y_1) = \s(y_2)\).  Then \(y_1=y_2\) because~\(V\) is a
  slice.  Then \(\s(x_1) = \rg(y_1) = \rg(y_2) = \s(x_2)\).
  This implies \(x_1 = x_2\) because~\(U\) is a slice.  This
  proves that the elements \(x\in U\) and \(y\in V\) with
  \(z=[x,y]\) are unique and that~\(\s|_{U\cdot V}\) is injective.
  Next, we assume instead that \(\Qu[x_1,y_1] = \Qu[x_2,y_2]\).
  This means that there is \(g\in \Gr\) with
  \(\s[x_1,y_1] = \rg(g)\) and \([x_1,y_1]\cdot g = [x_2,y_2]\).
  Equivalently, \(\s(y_1) = \rg(g)\) and
  \([x_1,y_1\cdot g] = [x_2,y_2]\).  This means that there is
  \(h\in\Gr[H]\) with \(\rg(h) = \s(x_1) = \rg(y_1)\) such that
  \((x_1\cdot h,h^{-1}\cdot y_1\cdot g) = (x_2,y_2)\).  Then
  \(x_1\cdot h = x_2\), so that \(\Qu(x_1) = \Qu(x_2)\) in
  \(\Bisp/\Gr[H]\).  This implies \(x_1 = x_2\) because~\(U\) is a
  slice.  Since the right \(\Gr[H]\)\nb-action on~\(\Bisp\) is
  free, it follows that \(h=1_{\s(x_1)}\).  Then
  \(y_1\cdot g = y_2\), and an analogous argument for the
  slice~\(V\) in the groupoid correspondence~\(\Bisp[Y]\) shows
  that \(y_1 = y_2\).  Therefore, the orbit space projection is
  injective on \(U\cdot V\).  This finishes the proof that
  \(U\cdot V\) is a slice in $\Bisp\Grcomp_{\Gr[H]}\Bisp[Y]$.
\end{proof}

\begin{proposition}
  \label{hcomp}
  Let $\Bisp\colon \Gr[K] \leftarrow \Gr[H]$ and
  $\Bisp[Y]\colon \Gr[H]\leftarrow \Gr$ be composable groupoid
  correspondences.  There is a well defined map
  \begin{align*}
    \mu^0_{\Bisp,\Bisp[Y]}\colon \ContS(\Bisp) \otimes
    \ContS(\Bisp[Y])
    &\to \ContS(\Bisp \Grcomp_{\Gr[H]} \Bisp[Y]),
    \\
    \mu^0_{\Bisp,\Bisp[Y]}(f_1\otimes f_2)([x,y])
    &=
    \sum_{\setgiven{h\in\Gr[H]^0}{\rg(h) = \s(x)}} f_1(x h)\cdot f_2(h^{-1} y).
  \end{align*}
  It extends to an isomorphism of
  $\Cst(\Gr[K])$-\(\Cst(\Gr)\)-correspondences
  \[
    \mu_{\Bisp,\Bisp[Y]}\colon
    \Cst(\Bisp) \otimes_{\Cst(\Gr[H])} \Cst(\Bisp[Y])
    \to \Cst(\Bisp\Grcomp_{\Gr[H]}\Bisp[Y]).
  \]
  This isomorphism is natural with respect to the \(2\)\nb-arrows
  in~\(\Grcat\).
\end{proposition}

\begin{proof}
  First examine~\(\mu^0_{\Bisp,\Bisp[Y]}\) on \(f_1\otimes f_2\)
  with \(f_1\in\Contc(U)\) and \(f_2\in\Contc(V)\) for slices
  \(U\subseteq \Bisp\) and \(V\subseteq \Bisp[Y]\).  Then
  \longref{Lemma}{lem:multiply_slices} implies that \(U\cdot V\)
  is a slice and that
  \(\mu^0_{\Bisp,\Bisp[Y]}(f_1 \otimes f_2) \in \Contc(U\cdot V)\),
  with the function given by \([x,y] \mapsto f_1(x) f_2(y)\) for
  \(x\in U\), \(y\in V\).  It follows that
  \(\mu^0_{\Bisp,\Bisp[Y]}\) maps
  \(\ContS(\Bisp) \otimes \ContS(\Bisp[Y])\) surjectively onto
  \(\ContS(\Bisp \Grcomp_{\Gr[H]} \Bisp[Y])\).

  Now give \(\ContS(\Bisp) \otimes \ContS(\Bisp[Y])\) the usual
  \(\ContS(\Gr[K])\)-\(\ContS(\Gr)\)-bimodule structure and the
  usual inner product, defined by
  \(\braket{f_1 \otimes f_2}{f_3 \otimes f_4} =
  \braket{f_3}{\braket{f_2}{f_3}*f_4}\).
  \longref{Lemma}{lem:corr_structure} implies easily that the
  map~\(\mu^0_{\Bisp,\Bisp[Y]}\) is an
  \(\ContS(\Gr[K])\)-\(\ContS(\Gr)\)-bimodule map that is formally
  isometric in the sense that
  \(\braket{\mu^0_{\Bisp,\Bisp[Y]}(F_1)}{\mu^0_{\Bisp,\Bisp[Y]}(F_2)}
  = \braket{F_1}{F_2}\) for all
  \(F_1,F_2\in \ContS(\Bisp) \otimes \ContS(\Bisp[Y])\).
  Then~\(\mu^0_{\Bisp,\Bisp[Y]}\) extends uniquely to the desired
  isometry~\(\mu_{\Bisp,\Bisp[Y]}\) on the norm completions.  The
  extension also remains an
  \(\ContS(\Gr[K])\)-\(\ContS(\Gr)\)-bimodule map.  The bimodule
  property extends by continuity to \(\Cst(\Gr[K])\) and
  \(\Cst(\Gr)\).  Since \(\Contc(U\cdot V)\) is in the range
  of~\(\mu^0_{\Bisp,\Bisp[Y]}\) for all slices \(U\) and~\(V\),
  the map~\(\mu_{\Bisp,\Bisp[Y]}\) is unitary.  It is easy to check
  that~\(\mu_{\Bisp,\Bisp[Y]}\) is natural for the \(2\)\nb-arrows
  in \(\Grcat\).
\end{proof}

\begin{theorem}
  \label{the:Grcat_to_Corr}
  The maps \(\Gr\mapsto \Cst(\Gr)\) on objects,
  \(\Bisp\mapsto \Cst(\Bisp)\) on arrows, and
  \(\alpha\mapsto \Cst(\alpha)\) on \(2\)\nb-arrows together with
  the identity maps \(\Cst(1_{\Gr}) = 1_{\Cst(\Gr)}\) and the maps
  \(\mu_{\Bisp,\Bisp[Y]}\colon \Cst(\Bisp) \otimes_{\Cst(\Gr[H])}
  \Cst(\Bisp[Y]) \to \Cst(\Bisp\Grcomp_{\Gr[H]}\Bisp[Y])\) form a
  strictly unital homomorphism of bicategories \(\Grcat\to \Corr\).
\end{theorem}

\begin{proof}
  The only reason to call this a theorem is that it summarises all
  the constructions in this section.  It remains to prove that the
  identity maps \(\Cst(1_{\Gr}) = 1_{\Cst(\Gr)}\) and the
  maps~\(\mu_{\Bisp,\Bisp[Y]}\) in \longref{Proposition}{hcomp} make the
  three diagrams commute that are required for homomorphisms of
  bicategories (see \cites{Benabou:Bicategories,
    Leinster:Basic_Bicategories}).  And this
  is routine to check, by testing equality of \(2\)\nb-arrows on
  elementary tensors where each factor is supported on a slice.
\end{proof}

Let~\(P\) be a monoid with unit element \(1\in P\).  We view~\(P\)
as a bicategory with only one object, set of arrows~\(P\), and only
identity \(2\)\nb-arrows.  A bicategory homomorphism from~\(P\)
to~\(\Grcat\) as defined in \cites{Benabou:Bicategories,
  Leinster:Basic_Bicategories} is described by the following data:
\begin{itemize}
  \item a groupoid~\(\Gr\);
  \item groupoid correspondences
    \(\Bisp_p\colon \Gr \leftarrow \Gr\) for
    \(p \in P\backslash\{1\}\);
  \item biequivariant homeomorphisms
    \(\sigma_{p,q}\colon \Bisp_p\Grcomp_{\Gr}\Bisp_q\to \Bisp_{pq}\)
    for \(p,q \in P\backslash\{1\}\);
\end{itemize}
we also let \(\Bisp_1 \defeq 1_{\Gr}\), and we let
\(\sigma_{1,q}\colon 1_{\Gr} \Grcomp_{\Gr} \Bisp_q \to \Bisp_q\) and
\(\sigma_{p,1}\colon \Bisp_p \Grcomp_{\Gr} 1_{\Gr} \to \Bisp_p\) be
the canonical maps in \longref{Lemma}{lem:unitor}.  The above data
describes a bicategory homomorphism if and only if the following
diagram commutes for all \(p,q,t\in P\):
\begin{equation}
  \label{eq:homomorphism_to_Grcat}
  \begin{tikzcd}[column sep=huge]
    \Bisp_p \Grcomp_{\Gr}\Bisp_q \Grcomp_{\Gr} \Bisp_t
    \arrow{r}{\id_{\Bisp_p}\Grcomp_{\Gr}\sigma_{q,t}}
    \arrow{d}[swap]{\sigma_{p,q}\Grcomp_{\Gr}\id_{\Bisp_t}}&
    \Bisp_p\Grcomp_{\Gr} \Bisp_{qt}
    \arrow{d}{\sigma_{p,qt}} \\
    \Bisp_{pq}\Grcomp_{\Gr} \Bisp_t
    \arrow{r}[swap]{\sigma_{pq,t}}  &
    \Bisp_{pqt}.
  \end{tikzcd}
\end{equation}
More generally, we may replace a monoid by a category.  Then
homomorphisms are defined in the same way.  The only change is that
we now have one groupoid~\(\Gr_x\) for each object \(x\in\Cat^0\),
and the groupoid correspondence~\(\Bisp_p\) for \(p\colon x\to y\)
becomes a correspondence  \(\Gr_y\leftarrow\Gr_x\).

Two homomorphisms of bicategories
\(\Cat \to \Cat[D] \to \Cat[E]\) may be composed in a
canonical way to a homomorphism \(\Cat \to \Cat[E]\)
(see~\cite{Benabou:Bicategories}).  Therefore, we may compose the
homomorphism described by the data above with the homomorphism in
\longref{Theorem}{the:Grcat_to_Corr} to get a homomorphism
\(P\to\Corr\).  Such homomorphisms are identified
in~\cite{Albandik-Meyer:Colimits} with product systems over the
monoid~\(P\).  The product system resulting from this composition
has the following data:
\begin{itemize}
\item the fibres are the \(\Cst\)\nb-algebra \(\Cst(\Gr)\) at
  \(1\in P\) and the underlying Banach spaces of the correspondences
  \(\Cst(\Bisp_p)\colon \Cst(\Gr) \leftarrow \Cst(\Gr)\) for
  \(p \in P\backslash\{1\}\);
\item the multiplication with \(\Cst(\Gr)\) and the
  inner product maps on the fibres of the product system come from
  the correspondence structure on \(\Cst(\Bisp_p)\);
\item the multiplication map
  \(\Cst(\Bisp_p)\otimes_{\Cst(\Gr)} \Cst(\Bisp_q) \to \Cst(\Bisp_{p
    q})\) for \(p,q\in P\setminus\{1\}\) comes from the isomorphism
  of correspondences
    \[
      \Cst(\Bisp_p)\otimes_{\Cst(\Gr)} \Cst(\Bisp_q)
      \xrightarrow{\mu_{\Bisp_p,\Bisp_q}}
      \Cst(\Bisp_p \Grcomp_{\Gr} \Bisp_q)
      \xrightarrow{\sigma_{p,q}} \Cst(\Bisp_{pq}).
    \]
\end{itemize}
This multiplication is associative because of the commuting
diagrams~\eqref{eq:homomorphism_to_Grcat} and the general theory of
bicategories.  In particular, it uses the commuting diagrams in the
definition of a bicategory homomorphism (see
\cites{Benabou:Bicategories, Leinster:Basic_Bicategories}), which involve the
maps~\(\mu_{p,q}\) in \longref{Proposition}{hcomp}.  Thus the
construction of a product system from a homomorphism \(P\to\Grcat\)
is an application of \longref{Theorem}{the:Grcat_to_Corr}.

Once we have got a product system, we may take its Cuntz--Pimsner
algebra to associate a \(\Cst\)\nb-algebra to a homomorphism
\(P\to \Grcat\).  The further study of this Cuntz--Pimsner algebra
becomes much easier if the composite homomorphism \(P\to\Corr\)
lands in the subbicategory of proper correspondences, that is,
correspondences where the left action is by compact operators.  The
following theorem characterises when the \(\Cst\)\nb-correspondence
\(\Cst(\Bisp)\) associated to a groupoid correspondence is proper:

\begin{theorem}
  \label{the:Cstar_X_proper}
  Let \(\Gr\) and~\(\Gr[H]\) be groupoids.  A groupoid
  correspondence \(\Bisp\colon \Gr[H]\leftarrow\Gr\) is proper if
  and only if the associated \(\Cst\)\nb-correspondence
  \(\Cst(\Bisp)\) is proper.
\end{theorem}

\begin{proof}
  Since \(\Cont_0(\Gr[H]^0)\) is embedded nondegenerately into
  \(\Cst(\Gr[H])\), a correspondence \(\Cst(\Gr[H])\leftarrow B\)
  for a \(\Cst\)\nb-algebra~\(B\) is proper if and only if
  \(\Cont_0(\Gr[H]^0)\) acts by compact operators.  In the proof of
  \longref{Lemma}{lem:positive_innerprod}, we have shown that the
  operators of pointwise multiplication by functions in
  \(\Contc(\Bisp/\Gr)\) on \(\Cst(\Bisp)\) are compact.  More
  precisely, if \(V\subseteq \Bisp\) is a slice and
  \(\varphi\in \Contc(\Qu(V))\), then we identified the operator of
  pointwise multiplication by \(\abs{\varphi}^2\) with a rank-one
  operator.  Then the \(\Cst\)\nb-completion \(\Cont_0(\Bisp/\Gr)\)
  of \(\Contc(\Bisp/\Gr)\) also acts by compact operators.  Since
  pointwise multiplication by \(\Cont_0(\Bisp/\Gr)\) is
  nondegenerate on \(\ContS(\Bisp)\), it remains nondegenerate on
  \(\Cst(\Bisp)\).  This implies that \(\Cont_0(\Bisp/\Gr)\) embeds
  nondegenerately into the \(\Cst\)\nb-algebra of compact operators
  on \(\Cst(\Bisp)\).  Therefore, pointwise multiplication with a
  function in \(\Contb(\Bisp/\Gr)\) is only compact if the function
  belongs to \(\Cont_0(\Bisp/\Gr)\).  The action of
  \(\Cont_0(\Gr[H]^0)\) on~\(\Cst(\Bisp)\) factors through the
  homomorphism \(\Cont_0(\Gr[H]^0) \to\Contb(\Bisp/\Gr)\) induced by
  \((\rg_{\Bisp})_*\colon \Bisp/\Gr \to \Gr[H]^0\).  This
  homomorphism factors through \(\Cont_0(\Bisp/\Gr)\) if and only if
  \((\rg_{\Bisp})_*\) is proper.  This finishes the proof.
\end{proof}

Tight groupoid correspondences are proper, of course, and therefore
induce proper \(\Cst\)\nb-correspondences.  It is unclear, however,
which further extra properties \(\Cst(\Bisp)\) has if~\(\Bisp\) is a
tight correspondence.

\begin{example}
  If~\(\Gr\) is Hausdorff, then the arrow space of the
  groupoid~\(\Gr\) with the usual left action and the source map as
  right anchor map is a groupoid correspondence
  \(\Bisp\colon \Gr \leftarrow \Gr^0\); here we view~\(\Gr^0\) as a
  groupoid with only identity arrows.  The resulting
  \(\Cst\)\nb-correspondence
  \(\Cst(\Bisp)\colon \Cst(\Gr) \leftarrow \Cst(\Gr^0) \cong
  \Cont_0(\Gr^0)\) is equivalent to a continuous family of
  representations of~\(\Cst(\Gr)\) on a continuous field of Hilbert
  spaces over~\(\Gr^0\).  This is the family of regular
  representations of~\(\Cst(\Gr)\) on the Hilbert spaces
  \(\ell^2(\Gr_x)\) for \(x\in\Gr^0\).  The homomorphism
  \(\Cst(\Gr) \to \Bound(\Cst(\Bisp))\) descends to a faithful
  representation of the \emph{reduced groupoid \(\Cst\)\nb-algebra}.
  In particular, it need not be faithful on~\(\Cst(G)\).  If~\(\Gr\)
  is not Hausdorff, then the above groupoid correspondence no longer
  exists because the orbit space \(\Gr/\Gr_0 = \Gr\) is not
  Hausdorff.  We may, however, pick \(x\in\Gr^0\) and construct a
  representation of~\(\Cst(\Gr)\) on the Hilbert space
  \(\ell^2(\Gr_x)\).  These representations, however, no longer fit
  together into a continuous field.  We may view the representation
  of~\(\Cst(\Gr)\) on \(\ell^2(\Gr_x)\) as a
  \(\Cst\)\nb-correspondence \(\Cst(\Gr) \leftarrow \C = \Cst(\pt)\)
  for the trivial groupoid~\(\pt\).  This comes from
  \(\Gr_x = \s^{-1}(\{x\}) \subseteq \Gr\), viewed as a groupoid
  correspondence \(\Gr_x\colon \Gr \leftarrow \pt\).
\end{example}

The above example shows that it may be hard to characterise when the
left action of \(\Cst(\Gr)\) on \(\Cst(\Bisp)\) for a groupoid
correspondence~\(\Bisp\) is faithful.

\section{Conduché fibrations as diagrams of groupoid
  correspondences}
\label{sec:Conduche}

Let~\(\Cat\) be a small category.  In this section, we examine
bicategory homomorphisms from~\(\Cat\) to~\(\Grcat\) that map each
object of~\(\Cat\) to a locally compact groupoid with only identity
arrows.  If \(\Gr_x\) and~\(\Gr_y\) are locally compact spaces, then
a groupoid correspondence \(\Gr_x\leftarrow \Gr_y\) is the same as a
\emph{topological correspondence} as defined
in~\cite{Albandik-Meyer:Product} (see
\longref{Example}{exa:topological_graph}).  The Cuntz--Pimsner
algebras of the product systems associated to diagrams of proper
topological correspondences are already studied
in~\cite{Albandik-Meyer:Product}, and we have nothing to add to
this, except that it is easy to generalise from diagrams defined
over monoids to diagrams defined over categories.  We would like to
point out, however, that the \(\Cst\)\nb-algebras of discrete
Conduché fibrations studied by Brown and
Yetter~\cite{Brown-Yetter:Conduche} are the special case of this
where the locally compact spaces in the diagram are all discrete.

\begin{definition}[\cite{Brown-Yetter:Conduche}]
  A \emph{discrete Conduché fibration} is a functor
  \(F\colon \Cat[E]\to\Cat\) with the following unique factorisation
  lifting property or \emph{Conduché condition}: if
  \(\varphi\colon y \to x\) is an arrow in~\(\Cat[E]\), then any
  factorisation
  \[
    \begin{tikzcd}
      F(y) \arrow[rr, bend right, "F(\varphi)"'] \arrow[r,"\lambda"] &
      z \arrow[r, "\varrho"] & F(x)
    \end{tikzcd}
  \]
  of~\(F(\varphi)\) in~\(\Cat\) lifts uniquely to a factorisation
  \[
    \begin{tikzcd}
      y \arrow[rr, bend right, "\varphi"'] \arrow[r,"\tilde\lambda"] &
      \tilde{z} \arrow[r, "\tilde\varrho"] & x
    \end{tikzcd}
  \]
  of~\(\varphi\) in~\(\Cat[E]\); here lifting means that~\(F\)
  maps~\(\tilde\lambda\) to~\(\lambda\) and~\(\tilde\varrho\)
  to~\(\varrho\).  A discrete Conduché fibration is
  \emph{row-finite} if for every object~\(X\) in~\(E\) and every
  morphism \(\beta\colon x\to F(X)\) in~\(\Cat\), the class of
  morphisms with target~\(x\) whose image under~\(F\) is~\(\beta\)
  is a ﬁnite set.
\end{definition}

\begin{definition}
  \label{def:Conduche_Cstar}
  Let~\(F\) be a row-finite discrete Conduché fibration.  The
  \(\Cst\)\nb-algebra of~\(F\) is the universal
  \(\Cst\)\nb-algebra~\(\Cst(F)\) generated by orthogonal
  projections~\(P_X\) for \(X\in \Cat[E]^0\) and partial
  isometries~\(S_\alpha\) for morphisms \(\alpha\colon Y\to X\)
  in~\(\Cat[E]\) that satisfy the following relations:
  \begin{enumerate}
  \item \label{en:Conduche_Cstar_1}%
    if \(X\neq Y\), then \(P_X P_Y = 0\);
  \item \label{en:Conduche_Cstar_2}%
    if \(\alpha\) and~\(\beta\) are composable, then
    \(S_{\alpha\beta} = S_\alpha S_\beta\);
  \item \label{en:Conduche_Cstar_3}%
    \(P_x = S_{\id_x}\) for all \(x\);
  \item \label{en:Conduche_Cstar_4}%
    if \(\alpha\colon y\to x\), then \(S_\alpha^* S_\alpha = P_y\);
  \item \label{en:Conduche_Cstar_5}%
    if \(F(\alpha) = F(\beta)\) and \(\alpha \neq \beta\), then
    \(S_\beta^* S_\alpha = 0\);
  \item \label{en:Conduche_Cstar_6}%
    if \(X\in \Cat[E]^0\) and \(g\colon x\to F(X)\) is an arrow
    in~\(\Cat\), then
    \[
      \sum_{F(\alpha)=g, r(\alpha)=X} S_\alpha S_\alpha^* = P_X.
    \]
  \end{enumerate}
\end{definition}

These definitions are inspired by the usual definition of a
(row-finite) higher-rank graph and its \(\Cst\)\nb-algebra: this is
the special case when~\(\Cat\) is the monoid~\((\N^k,+)\).  We will
not say more about the theory of discrete Conduché fibrations and
refer to~\cite{Brown-Yetter:Conduche} for further discussion.  The
theorem below describes them in another way, using the bicategory of
groupoid correspondences.  We believe that this alternative
description is much more informative.  From our point of view, a
discrete Conduché fibration is equivalent to an action of~\(\Cat\)
on discrete sets by correspondences.  In particular, a higher-rank
graph is an action of~\(\N^k\) on a discrete set by correspondences.

To be precise, Brown and Yetter define \(\Cst(F)\) only if~\(F\) is
row-finite and ``strongly surjective''.  This says that the sum on
the right in condition~\ref{en:Conduche_Cstar_6} is non-empty for
all \(X\) and~\(g\).  If this fails, the condition asks for
\(P_X=0\), which may entail further generators to become~\(0\) and
may reduce the whole \(\Cst\)\nb-algebra to be~\(0\).  We allow this
degenerate case, however, because the following theorem remains
true.

The following theorem uses product systems over categories and their
absolute Cuntz--Pimsner algebras.  All this is usually considered
only over monoids, but the definitions make perfect sense over a
category instead.  The quickest way for us to define a product
system over the category~\(\Cat\) is as a bicategory homomorphism
\(\Cat\to\Corr\).  We call a product system proper if this
homomorphism lands in the subbicategory of proper correspondences.
The \emph{absolute Cuntz--Pimsner algebra} of a proper product
system is defined in \cite{Albandik-Meyer:Colimits}*{Definition~6.8}.
Over a monoid, this definition is the usual one, asking the
Cuntz--Pimsner covariance condition for all elements of the
underlying \(\Cst\)\nb-algebra and not just on some ideal.

\begin{theorem}
  \label{the:Conduche}
  A discrete Conduché fibration \(F\colon \Cat[E] \to \Cat\) is
  equivalent to a bicategory homomorphism
  \(\tilde{F}\colon \Cat\to\Grcat\) with the extra property that
  each groupoid~\(\Gr_x\) for \(x\in\Cat^0\) is a discrete set with
  only identity arrows.  Here ``equivalent'' means a bijection on
  isomorphism classes.  The fibration~\(F\) is row-finite if and
  only if the corresponding homomorphism~\(\tilde{F}\) is proper.
  If~\(F\) is row-finite, then \(\Cst(F)\) is naturally isomorphic
  to the absolute Cuntz--Pimsner algebra of the product system
  over~\(\Cat\) associated to~\(\tilde{F}\).
\end{theorem}

\begin{proof}
  For \(x\in\Cat^0\) and \(g\in\Cat\), let
  \(\Gr_x \defeq F^{-1}(x) \subseteq \Cat[E]^0\) and
  \(\Bisp_g \defeq F^{-1}(g) \subseteq \Cat[E]\), respectively.  So
  \(\Cat[E]^0 = \bigsqcup_{x\in\Cat^0} \Gr_x\) and
  \(\Cat[E] = \bigsqcup_{g\in\Cat} \Bisp_g\).  We view~\(\Gr_x\) for
  \(x\in\Cat^0\) as an étale groupoid with only identity arrows and
  the discrete topology.  The range and source maps
  \(\rg,\s\colon \Cat[E] \rightrightarrows \Cat[E]^0\) restrict to
  maps \(\rg\colon \Bisp_g \to \Gr_{\rg(g)}\) and
  \(\s\colon \Bisp_g \to \Gr_{\s(g)}\) because~\(F\) is a functor.
  These make~\(\Bisp_g\) a groupoid correspondence
  \(\Gr_{\rg(g)} \leftarrow \Gr_{\s(g)}\); the conditions for a
  groupoid correspondence are obviously satisfied because our
  groupoids have only identity arrows and~\(\Bisp_g\) has the
  discrete topology.  Since~\(F\) is a functor, the composition
  in~\(\Cat[E]\) restricts to maps
  \(\mu_{g,h}\colon \Bisp_g \times_{\Gr_{\s(g)}} \Bisp_h \to
  \Bisp_{g\cdot h}\) for all \(g,h\in\Cat\) with \(\s(g) = \rg(h)\).
  Since each~\(\Gr_x\) is a space,
  \(\Bisp_g \times_{\Gr_{\s(g)}} \Bisp_h = \Bisp_g \Grcomp_{\Gr_{\s(g)}}
  \Bisp_h\).  The discrete Conduché fibration condition says exactly
  that each map~\(\mu_{g,h}\) is bijective, as required for a
  homomorphism of bicategories.  The maps~\(\mu_{g,h}\) determine
  the multiplication in~\(\Cat[E]\) because the set of composable
  pairs of arrows in~\(\Cat[E]\) is the disjoint union of the fibre
  products \(\Bisp_g \times_{\Gr_{\s(g)}} \Bisp_h\) for
  \((g,h)\in\Cat^2\).  The associativity of the multiplication
  in~\(\Cat[E]\) is equivalent to the associativity condition for a
  bicategory homomorphism in~\eqref{eq:homomorphism_to_Grcat}.

  If \(x\in\Cat^0\subseteq \Cat\), then \(\Bisp_x = \Gr_x\) is the
  set of all unit arrows on objects \(\hat{x}\in\Cat[E]^0\) with
  \(F(\hat{x})=x\) by \cite{Brown-Yetter:Conduche}*{Lemma~2.3}, as
  required for a groupoid homomorphism.  Since all arrows
  in~\(\Gr_x\) are identities, the multiplication maps~\(\mu_{g,h}\)
  when \(g\) or~\(h\) is a unit arrow in~\(\Cat\) contain no
  information.  So a discrete Conduché fibration
  \(\Cat[E] \to \Cat\) gives a homomorphism of bicategories
  \(\Cat\to\Grcat\) with the extra property that each~\(\Gr_x\) is a
  discrete set viewed as an étale groupoid with the discrete
  topology and only identity arrows.  This construction is
  reversible, that is, any bicategory homomorphism with this extra
  property comes from a discrete Conduché fibration, which is unique
  up to isomorphism.

  The correspondence~\(\Bisp_g\) for an arrow~\(g\) in~\(\Cat\) is
  proper if and only if for each object \(x\in \Cat[E]^0\), the set
  of \(\gamma\in\Bisp_g\) with \(r(\gamma)=g\) is finite.  This
  happens for all~\(g\) if and only if the fibration~\(F\) is
  row-finite.

  Since the spaces \(\Gr_x\) and~\(\Bisp_g\) are discrete, the
  \(\Cst\)\nb-algebras and \(\Cst\)\nb-\alb{}correspondences in our
  proper product system have obvious bases, namely, the
  delta-functions of elements in \(\Gr_x\) and~\(\Bisp_g\),
  respectively.  Let \(\pi_g\colon \Bisp_g \to \Bound(\Hils)\) for
  \(g\in \Cat\) be a Cuntz--Pimsner covariant representation of the
  proper product system \(\Cst(\Bisp_g)_{g\in\Cat}\) over~\(\Cat\).
  Since the delta-functions span dense subspaces in \(\Cst(\Gr_x)\)
  and \(\Cst(\Bisp_g)\), the representation
  \(\pi=(\pi_g)_{g\in\Cat}\) of the product system is determined by
  its values on the delta-functions.  Actually, since
  \(\Bisp_{1_x} = \Gr_x\) for a unit arrow, there is no need to list
  the generators for objects at all.  So we only need the values
  \(S_\alpha \defeq \pi_f(\delta_\alpha)\) if \(f\in \Cat(x,y)\),
  \(x,y\in \Cat^0\), and \(\alpha\in\Bisp_f\).  It is convenient to
  also define \(P_X \defeq S_{1_X} = \pi_{1_x}(\delta_X)\) if
  \(x\in \Cat^0\) and \(X\in\Gr^0_x\).  Actually,
  \(\alpha\in \bigsqcup_{g\in\Cat} \Bisp_g = \Cat[E]\) and
  \(X\in \bigsqcup_{x\in\Cat^0} \Gr^0_x = \Cat[E]^0\).  It suffices
  to check the conditions for a Cuntz--Pimsner covariant
  representation of the product system on the basis of
  delta-functions.  We claim that a
  family~\((S_\alpha)_{\alpha\in\Cat[E]}\) as above defines a
  representation if and only if
  \begin{enumerate}
  \item \label{Conduche:1}%
    \(P_X \defeq S_{1_X}\) for
    \(X\in \bigsqcup_{x\in\Cat^0} \Gr_x = \Cat[E]^0\) are
    mutually orthogonal projections;
  \item \label{Conduche:2}%
    each~\(S_\alpha\) is a partial isometry;
  \item \label{Conduche:3}%
    for all \(\alpha\in\Cat[E]\),
    \(S_\alpha^* S_\alpha = P_{\s(\alpha)}\);
  \item \label{Conduche:4}%
    if \(\alpha,\beta\in\Cat[E]\) are different, then
    \(S_\alpha^* S_\beta = 0\);
  \item \label{Conduche:5}%
    if \(\alpha\) and~\(\beta\) are composable in
    \(\Cat[E] \defeq \bigsqcup_{g\in\Cat} \Bisp_g\), then
    \(S_\alpha S_\beta = S_{\alpha\beta}\).
  \item \label{Conduche:6}%
    if \(x,y\in \Cat^0\), \(X\in\Gr_x^0\), \(g\in\Cat(y,x)\), then
    \(P_X = \sum_{\alpha\in \Bisp_g, \rg(\alpha) = X} S_\alpha
    S_\alpha^*\).
  \end{enumerate}
  Condition~\ref{Conduche:1} says that \(\delta_X\mapsto P_X\)
  extends to a \Star{}homomorphism on
  \(\Cont_0\bigl(\bigsqcup_{x\in\Cat^0} \Gr_x^0\bigr) =
  \Cont_0(\Cat[E]^0)\).  Condition~\ref{Conduche:2} follows from
  \ref{Conduche:3} and~\ref{Conduche:4}, and these say that
  \(\pi_{\s(g)}(\braket{\delta_\alpha}{\delta_\beta}) =
  \pi_g(\delta_\alpha)^* \pi_g(\delta_\beta)\) for all \(g\in\Cat\),
  \(\alpha,\beta\in\Bisp_g\); this is the condition for the right
  inner product in a Toeplitz representation of a product system.
  Condition~\ref{Conduche:5} says that
  \(\pi_g(S_\alpha) \pi_h(S_\beta) = \pi_{g h}(S_{\alpha \beta})\)
  for all \((g,h)\in\Cat^2\) and \(\alpha\in \Bisp_g\),
  \(\beta\in \Bisp_h\); this is the condition for the multiplication
  in a Toeplitz representation of a product system for \(g,h\) not a
  unit.  If \(g\) or~\(h\) is a unit, then the conditions
  \(P_X S_\alpha = \delta_{X,\rg(\alpha)} S_\alpha\) and
  \(S_\alpha P_X = \delta_{X,\s(\alpha)} S_\alpha\) follow from
  \ref{Conduche:1}, \ref{Conduche:3} and~\ref{Conduche:6}.
  Condition~\ref{Conduche:6} says that say the representation
  of~\(\Cst(\Bisp_g)\) for an arrow \(g\in\Cat(y,x)\) is
  Cuntz--Pimsner covariant on all elements of \(\Cst(\Gr_x)\); by
  assumption, the product system is proper, so that \(\Cst(\Gr_x)\)
  acts on~\(\Cst(\Bisp_g)\) by compact operators.  Now it is easy to
  see that the resulting list of conditions is equivalent to the
  list in \longref{Definition}{def:Conduche_Cstar}.  Therefore, both
  universal properties define the same \(\Cst\)\nb-algebra.
\end{proof}

We now specialise this to \(k\)\nb-graph \(\Cst\)\nb-algebras
(see~\cite{Kumjian-Pask:Higher_rank}):

\begin{corollary}
  A \(k\)\nb-graph is equivalent to a bicategory homomorphism from
  \((\N^k,+)\) viewed as a bicategory to~\(\Grcat\) that maps the
  unique object of the bicategory~\(\N^k\) to a discrete set.  The
  \(k\)\nb-graph is row-finite if and only if it corresponds to a
  homomorphism to~\(\Grcat_\prop\).  Then the \(\Cst\)\nb-algebra of
  the \(k\)\nb-graph is naturally isomorphic to the absolute
  Cuntz--Pimsner algebra of the resulting proper product system
  over~\(\N^k\).
\end{corollary}

\begin{proof}
  This follows from \longref{Theorem}{the:Conduche} because discrete
  Conduché fibrations with \(\Cat=(\N^k,+)\) are the same as
  rank-\(k\) graphs (see~\cite{Brown-Yetter:Conduche}).
\end{proof}

We may also allow each~\(\Gr_x\) to be a locally compact space, made
a groupoid with only identity arrows.  We propose this as the right
locally compact version of a Conduché fibration.  Now the category
\(\Cat[E] = \bigsqcup_{g\in\Cat} \Bisp_g\) is a locally compact
topological category with a continuous functor
\(F\colon \Cat[E]\to\Cat\) that is a discrete Conduché fibration as
defined above when we forget the topologies.  In addition, the
source map \(\s\colon \Cat[E]\to\Cat[E]^0\) and the multiplication
in~\(\Cat[E]\) are local homeomorphisms.  These extra conditions are
necessary and sufficient for a functor \(F\colon \Cat[E]\to\Cat\) to
come from a bicategory homomorphism \(\Cat\to\Grcat\) by the
construction above.

If~\(\Gr_x\) are locally compact spaces, then a
groupoid correspondence \(\Gr_x\leftarrow \Gr_y\) is the same as a
\emph{topological correspondence} as defined
in~\cite{Albandik-Meyer:Product} (see
\longref{Example}{exa:topological_graph}).  When~\(\Cat\) is a
monoid, we get actions of that monoid on a topological space by
topological correspondences exactly as
in~\cite{Albandik-Meyer:Product}.  If \(\Cat=(\N^k,+)\), we get the
topological rank-\(k\) graphs by
Yeend~\cite{Yeend:Topological-higher-rank-graphs}.  The topological
analogue of row-finiteness says exactly that these topological
correspondences are proper.  If they are also all surjective, then
the \(\Cst\)\nb-algebra of the topological rank-\(k\) graph is
isomorphic to the absolute Cuntz--Pimsner algebra of the proper
product system over~\(\N^k\) defined by the homomorphism
\(\N^k \to \Grcat \to \Corr\).

We may also replace a discrete set by a discrete group~\(G\) or a
transformation group \(V\rtimes \Gamma\) for an action of a
group~\(\Gamma\) on a discrete set~\(V\).  We have related groupoid
correspondences on \(V\rtimes \Gamma\) to self-similar graphs in
\longref{Example}{exa:Exel-Pardo}.  In an analogous way, a
bicategory homomorphism \(\N^k \to \Grcat_\prop\) that maps the
unique object of the category~\(\N^k\) to \(V\rtimes \Gamma\) is
equivalent to a self-similar (row-finite) \(k\)\nb-graph as defined
by Li and Yang~\cite{Li-Yang:Self-similar_k-graph}.  They also
describe the \(\Cst\)\nb-algebra of such an object as the
Cuntz--Pimsner algebra of a proper product system over~\(\N^k\).
For each $n \in \N^k$ and $\mu \in d^{-1}(n)$, the map
\(\delta_{\mu,g} \mapsto \xi_\mu j(g)\) implements an isomorphism
from \(\Cst(\Bisp)\) to the \(\Cst\)\nb-correspondence that is
denoted by \(X_{G,\Lambda,n}\)
in~\cite{Li-Yang:Self-similar_k-graph}. Then the product system by
Li and Yang is naturally isomorphic to the product system
corresponding to the composite homomorphism
\(\N^k \to \Grcat \to \Corr\).

Summing up, various constructions of \(\Cst\)\nb-algebras from
combinatorial or dynamical data may be realised as Cuntz--Pimsner
algebras of product systems associated to homomorphisms to the
bicategory of groupoid correspondences.  This interprets these
\(\Cst\)\nb-algebras as covariance algebras of generalised dynamical
systems, where a category acts on an étale groupoid by groupoid
correspondences.

\end{document}